\documentclass[12pt]{amsart}
\usepackage{amsmath,amssymb,amsfonts,amsthm,amstext}
\usepackage{url,xspace}
\usepackage{color,graphicx}
\usepackage{graphics}
\usepackage{enumerate}
\usepackage{cleveref}

\usepackage[margin=3cm]{geometry}

\usepackage{amssymb}

\newcommand{\ER} {Erd\H{o}s-R\'enyi }

\newcommand{\intersect}{\cap}
\newcommand{\union}{\cup}

\newcommand{\old}[1]{}

\newcommand{\p}{\mathbf{p}}
\renewcommand{\P}{\mathbb{P}}
\newcommand{\E}{{\mathbb E}}

\newcommand{\df}{\textbf}

\newcommand{\fp}{\mathsf{FP}}

\newcommand{\cN}{\mathcal{N}}
\newcommand{\cP}{\mathcal{P}}

\newcommand{\vn}{\mathcal{N}}
\newcommand{\vp}{\mathcal{P}}
\newcommand{\vd}{\mathcal{D}}
\newcommand{\vnm}{\widetilde{\mathcal{N}}}
\newcommand{\vpm}{\widetilde{\mathcal{P}}}
\newcommand{\vdm}{\widetilde{\mathcal{D}}}
\newcommand{\vei}{\mathcal{E}^{(1)}}
\newcommand{\vsii}{\mathcal{S}^{(2)}}
\newcommand{\veii}{\mathcal{E}^{(2)}}
\newcommand{\vsi}{\mathcal{S}^{(1)}}

\newcommand{\bp}{\mathbf{p}}
\newcommand{\bq}{\mathbf{q}}
\newcommand{\cA}{\mathcal{A}}
\newcommand{\cB}{\mathcal{B}}
\newcommand{\cC}{\mathcal{C}}
\newcommand{\cS}{\mathcal{S}}
\newcommand{\cQ}{\mathcal{Q}}
\newcommand{\G}{\mathcal{G}}
\newcommand{\T}{\mathcal{T}}

\newcommand{\n}{{N}}
\renewcommand{\p}{{P}}
\renewcommand{\d}{{D}}
\newcommand{\nm}{\widetilde{{N}}}
\renewcommand{\pm}{\widetilde{{P}}}
\newcommand{\dm}{\widetilde{{D}}}
\newcommand{\ei}{{E}^{(1)}}
\newcommand{\sii}{{S}^{(2)}}
\newcommand{\eii}{{E}^{(2)}}
\newcommand{\si}{{S}^{(1)}}

\newcommand{\tildebp}{\widetilde{\bp}}
\newcommand{\tbp}{\widehat{\bp}}

\newcommand{\tH}{\widehat{H}}
\newcommand{\tF}{\widehat{F}}
\newcommand{\tN}{\widetilde{N}}
\newcommand{\tP}{\widetilde{P}}
\newcommand{\tD}{\widetilde{D}}
\newcommand{\tx}{\tilde{x}}

\newcommand{\tp}{\tilde{p}}

\newtheorem{thm}{Theorem}
\newtheorem{lemma}[thm]{Lemma}
\newtheorem{prop}[thm]{Proposition}
\newtheorem{cor}[thm]{Corollary}

\theoremstyle{definition}

\crefname{thm}{Theorem}{Theorems}
\crefname{lemma}{Lemma}{Lemmas}
\crefname{prop}{Proposition}{Propositions}
\crefname{cor}{Corollary}{Corollaries}
\crefname{section}{Section}{Sections}
\crefname{table}{Table}{Tables}
\crefname{definition}{Definition}{Definitions}
\crefname{example}{Example}{Examples}
\crefname{figure}{Figure}{Figures}

\newcounter{mycount}

\title{Galton-Watson Games}
\author{Alexander E. Holroyd}
\address{A. E. Holroyd}
\email{holroyd@uw.edu}

\author{James B. Martin}
\address{J. B. Martin, Department of Statistics,
University of Oxford}
\email{martin@stats.ox.ac.uk}

\date{April 8, 2019}

\begin{document}
\begin{abstract}We consider two-player combinatorial games in which the graph of positions is random and perhaps infinite,
focusing on directed Galton-Watson trees.
As the offspring distribution is varied, a
game can undergo a phase transition, in which the probability of a draw under optimal play becomes positive.
We study the nature of the phase transitions which
occur for normal play rules (where a player unable to move loses the game) and mis\`ere rules (where a player unable to move wins), as well as for an ``escape game'' in which one
player tries to force the game to end while the other tries to prolong it forever.  For instance, for a Poisson$(\lambda)$
offspring distribution, the game tree is infinite with
positive probability as soon as $\lambda>1$, but the game
with normal play has positive probability of draws if and
only if $\lambda>e$. The three games generally have
different critical points; under certain assumptions
the transitions are continuous for the normal and mis\`ere games and discontinuous for the escape game, but we also
discuss cases where the opposite possibilities occur. We connect the nature of the phase transitions to the behaviour
of quantities such as the expected length of the game under
optimal play. We also establish inequalities relating the games to each other; for instance, the
probability of a draw is at least as great in the mis\`ere game as in the normal game.
\end{abstract}

\maketitle

\section{Introduction}
Game theory naturally often focuses on carefully chosen games for which interesting mathematical analysis is possible.  What can be said about games in the wild?  One approach to this question is to consider games whose rules are \emph{typical}, i.e.\ chosen at random, although known to the players.  In this article we consider rules arising from random trees.

We consider combinatorial games whose positions and moves are described by a directed
acyclic graph $\G$.  A token is located at a vertex, and the two players take
turns to move it along a directed edge to a new vertex.  In the \df{normal
game}, a player \df{loses} the game if they cannot move (that is, if the
token is at a vertex with outdegree zero), and the other player \df{wins}.


We are interested in optimal play.  Thus, a \df{strategy} for a particular
player is a map that assigns a legal move for that player (where one exists)
to every vertex.  For a given starting vertex for the token, a strategy is
\df{winning} if it yields a win for that player, no
matter what strategy the other player uses.  Fix a starting vertex.  If $\G$
is finite, then it is easily seen that exactly one player has a winning
strategy; we then say that the game is a \df{win} for that player (and a
\df{loss} for the other).  More interestingly, if $\G$ is infinite, then it
is possible that neither player has a winning strategy, in which case we say
that the game is a \df{draw}.

We also consider two other rules for determining the game outcome. In the
\df{mis\`ere game}, a player \emph{wins} if they cannot move.  In the
\df{escape game}, the two players have distinct goals. One designated player,
called \df{Stopper}, wins if either player is ever unable to move, in which case
the other player, \df{Escaper}, loses.  If Stopper has no winning
strategy then the game is said to be a win for Escaper.

In a sense there is no loss of generality in assuming that $\G$ is a directed tree: if not, every game position may be augmented with a record of the sequence of moves that led to it; these augmented positions then form a tree.

We focus on Galton-Watson trees.  Thus, let
$\G=\T$ be the graph of a Galton-Watson branching process of offspring
probability mass function $\bp=(p_0,p_1,\ldots)$, with directed edges from
parents to children.  Let the token start at the root vertex $o$. We emphasize that
although the graph is random, it is assumed known to both players when
deciding on their strategies.

Let $N=N(\bp)$ be the probability that the normal game is a
win for the first (``Next'') player, let $P=P(\bp)$ be the
probability that it is a win for the second (``Previous'')
player, and let $D=1-N-P$ be the probability that it is a
draw.  Let $\tN,\tP,\tD$ be the analogous probabilities for
the mis\`ere game. For the escape game, let $\si$
(respectively, $\sii$) be the probabilities that the
stopper wins assuming the stopper has the first (respectively,
second) move.  Similarly let $\ei=1-\sii$ and $\eii=1-\si$ be
the win probabilities for the escaper when moving first or
second respectively.

It is well known that the Galton-Watson process exhibits a phase transition:
the process survives (i.e.\ $\T$ is infinite) with positive probability if
and only $\mu>1$ (or $p_1=1$), where $\mu:=\sum_i ip_i$ is the mean of the
offspring distribution. However, survival is not sufficient for
the existence of a draw -- intuitively, that requires
not just an infinite path, but an infinite path that
neither player can profitably deviate from.
Indeed, we will find that the draw and
escape probabilities $D,\tD,\ei,\eii$ undergo phase transitions as $\bp$ is varied, but typically \emph{not} at the same location as the survival phase transition.

The model can be analyzed in terms of generating functions. Let
$G(x)=G_{\bp}(x):=\sum_{i=0}^\infty p_i x^i$ be the generating function of the
offspring distribution. It is also convenient to define the functions $F=1-G$
and $H=1-G+p_0$. We denote iterates of functions by superscripts:
$F^2(x)=(F\circ F)(x)=F(F(x))$, etc.  Let $\fp(f)=\fp_{[0,1]}(f):=\{x\in[0,1]: f(x)=x\}$ denote
the set of fixed points of a function $f$ in the interval $[0,1]$.

\begin{thm}[Fixed points]\label{main}
For the normal, mis\`ere, and escape games played on a Galton-Watson tree
with offspring distribution $\bp$, we have:
\begin{enumerate}[\rm (i)]
  \item $D=\max \fp(F^2) - \min \fp(F^2);\;\; N=\min \fp(F^2);\;\;
  P=1-\max\fp(F^2)$;
  \item $\tD=\max \fp(H^2) - \min \fp(H^2);\;\; \tN=\min \fp(H^2);\;\;
  \tP=1-\max\fp(H^2)$;
  \item $\ei=\max \fp(F\circ H);\;\; \eii=1-\min \fp(H\circ F)$.
\end{enumerate}
\end{thm}
Note for instance that $D>0$ if and only if $F^2$ has
multiple fixed points in $[0,1]$.

Next we examine how the three games are related to each other.
It turns out that several
inequalities hold. Some are obvious, others more surprising.
In the following, $a,b\leq c$ means that $a\leq c$ and
$b\leq c$.
\begin{thm}[Inequalities]\label{ineq}
For a Galton-Watson process with any fixed offspring
distribution, we have:
\begin{enumerate}[\rm (i)]
\item $\n,\nm\leq \si$; \quad $\p,\pm\leq \sii$;
\item $\sii\leq \si$; \quad $\pm\leq\nm$;
\item $\pm \leq \p,\n$; \quad $\d\leq \dm$.
\end{enumerate}
Besides these inequalities and those implied by them, no
other inequalities between pairs of the $10$ outcome
probabilities hold in general.
\end{thm}
The classification into parts (i)--(iii) in Theorem~\ref{ineq} reflects
different types of argument. The inequalities in (i) follow from simple
implications that hold on any directed acyclic graph; for example, if the
first player can force the game to end after an odd number of moves then she
can of course force it to end.  Those in (ii) come from strategy-stealing
arguments involving the (distributional) homogeneity of the Galton-Watson
tree: if the first player opens with a random move then the resulting
position has the same law as before.  The inequalities in (iii) are proved by analytic methods, and we lack intuitive explanations for them.  The last
inequality is perhaps the most striking: draws are at
least as likely in the mis\`ere game as in the normal game.

Now we describe some examples of phase transitions that arise as the offspring distribution is varied.

%
\begin{samepage}
\begin{prop}[Examples]
\label{examples}\ 
\begin{enumerate}[\rm (i)]
  \item \emph{Binary branching.}  Let
      $(p_0,p_1,p_2)=(1-t,0,t)$ for $t\in[0,1]$, and note
      that the probability of survival is positive if and
      only if $t>1/2$.  The normal game draw probability
      $D$ has a phase transition at
      $t_n:=\sqrt3/2=0.866\ldots$, in the sense that $D>0$ if and only if
      $t>t_n$.  The transition is \emph{continuous}:
      $D\to 0$ as $t\downarrow t_n$.  Similarly, the
      mis\`ere draw probability $\tD$ has a continuous
      phase transition at $t_m:=3/4$.  In contrast, the
      escape game has a \emph{discontinuous} phase
      transition at $t_e:=3/2^{5/3}=0.945\ldots$: $\ei$
      is positive if and only if $\eii$ is positive, which
      happens if and only if $t\geq t_e$.  In fact
      $\ei=2^{4/3}/3=0.840\ldots$ at $t=t_e$.
  \item \emph{Poisson offspring.}  Let the offspring distribution be
      Poisson with mean $\lambda$, and note that the survival probability
      is positive if and only if $\lambda>1$.  The normal and mis\`ere
      games have continuous phase transitions at $\lambda_n=e$ and
      $\lambda_m=2.103\ldots$ respectively (where the latter is the solution of $\lambda=e^{\lambda(1-e^{-\lambda})}$): the draw probability is positive if
      and only if $\lambda$ exceeds the respective threshold.  The escape
      game has a discontinuous phase transition at
      $\lambda_e=3.319\ldots$.
  \item \emph{Geometric offspring.}  Let $p_i=(1-\alpha)
      \alpha^i$ for $i\geq 0$.  The draw and escape
      probabilities $D,\tD,\ei,\eii$ are zero for all
      $\alpha \in(0,1)$.
\end{enumerate}
\end{prop}
\end{samepage}

\begin{figure}
\begin{center}
\includegraphics[width=0.49\textwidth]{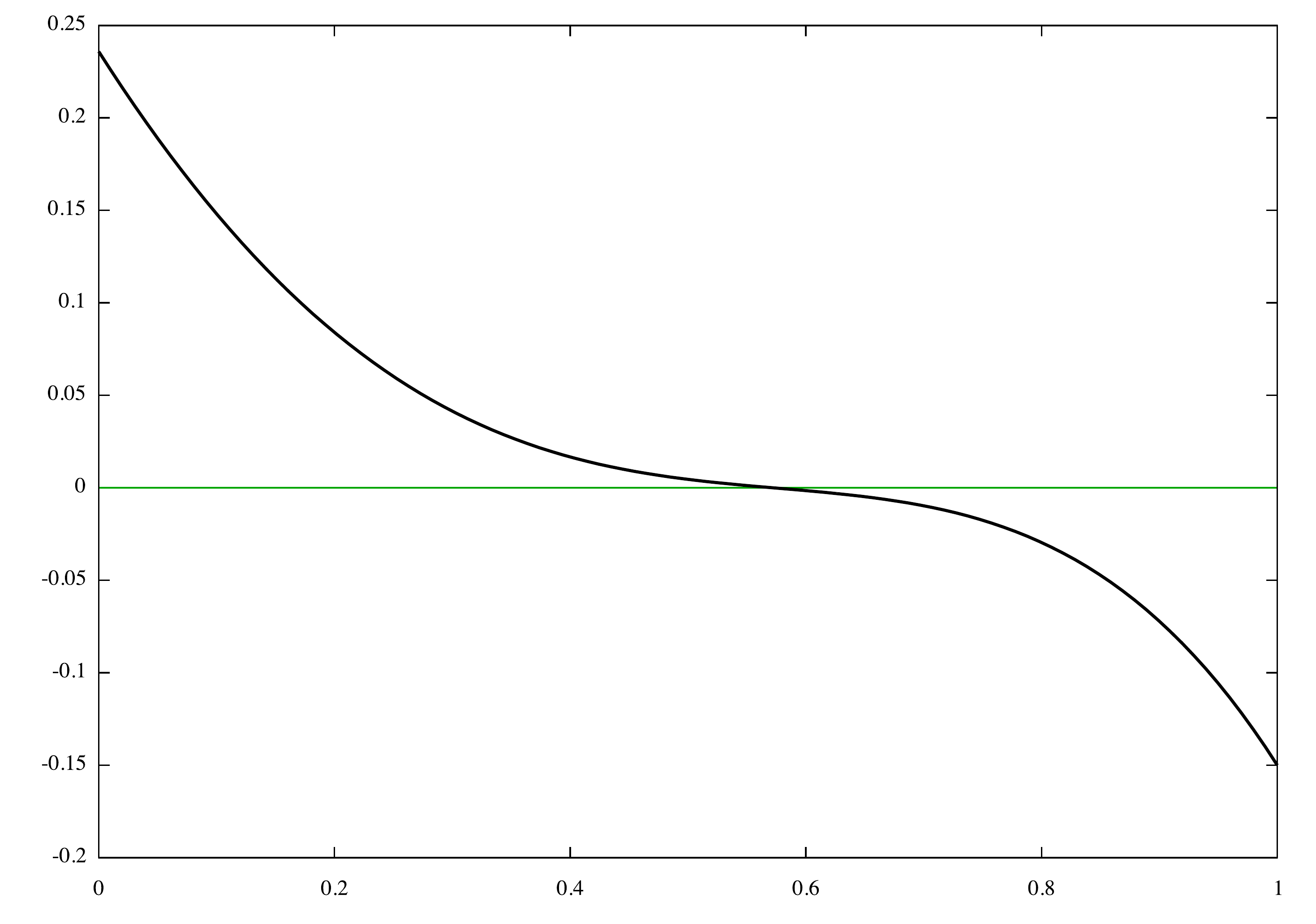}
\includegraphics[width=0.49\textwidth]{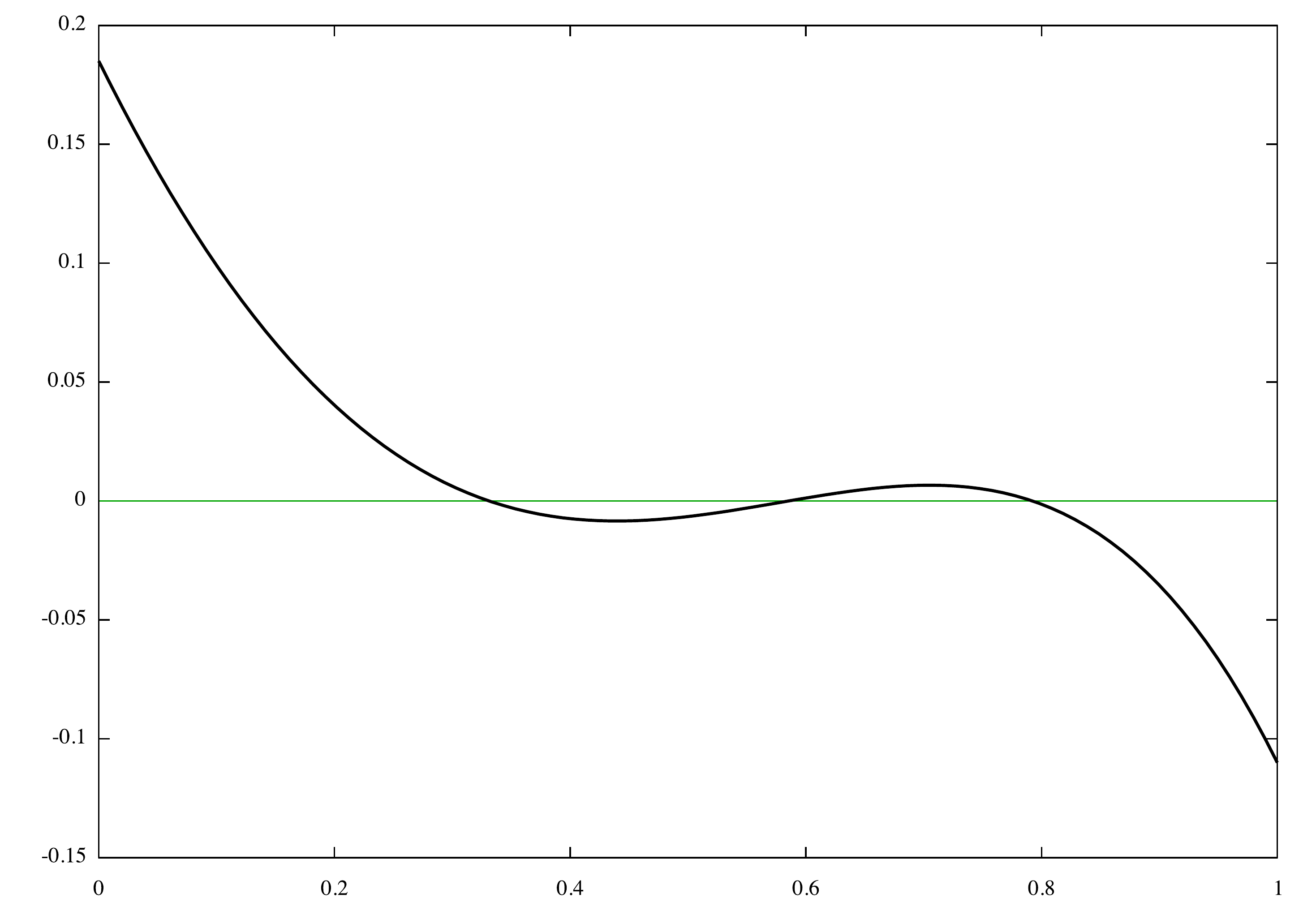}
\end{center}\label{binarynormalfigure}
\caption{
The function $F^2(x)-x$ for the binary branching distribution of \cref{examples}(i). The roots of this function are the
elements of $\fp(F^2)$.
On the left, $p=0.85$ (sub-critical for the normal game) -- the function has a unique root and the probability of a draw
in the normal game is 0. On the left, $p=0.89$ -- the function has three roots and
the probability of a draw is the distance from the smallest to the largest root.
As $p$ passes through the critical point, the two new roots emerge continuously from
the existing root.
For a contrasting example with a discontinuous phase transition,
see Figure \ref{discontinuousnormalfigure} in Section
\ref{sec-examples}.
}
\end{figure}

Note that the draw probability $D$ is not in general monotone in the
offspring distribution: the geometric distribution in (iii) stochastically
dominates the binary branching distribution in (i) if $\alpha$ is small
enough as a function of $t$, but the former has $D=0$ while the latter has
$D>0$ (for suitable $t$).  Similar remarks apply to $\tD$, $\ei$ and $\eii$.

See Figure \ref{binarynormalfigure} for an illustration
of Theorem \ref{main}(i) in the binary branching case of
Proposition \ref{examples}(i).

Theorem~\ref{main} enables the games to be analyzed for many other offspring
distributions: the outcome probabilities are given in terms of solutions of
equations (although not always as closed-form expressions).  Another interesting case (which we do not treat in detail) is the Binomial$(n,p)$ distribution, under which $\T$ can be viewed as the percolation cluster on a regular tree.  Here the normal game has draws if and only if $p>(n+1)^{n-1}/n^n$.

We typically find that phase transitions are continuous for the normal and mis\`ere games and discontinuous for the escape game, as in the above
examples.  However, we can concoct examples with the opposite behavior, as
well as more exotic phase transitions, as follows.

\begin{prop}[Exotic Examples]
\label{exotic}
  For each of (i)--(iii) below there exists a
continuous family $(\bp(t):t\in(0,1))$ of offspring distributions,
of uniformly bounded support, with the given
properties.
\begin{enumerate}[\rm (i)]
\item The normal game has a
(non-trivial) discontinuous phase transition: there exists
    $t^*\in(0,1)$ such that $D=0$ for $t<t^*$ while $D\in(0,1)$ for
    $t\geq t^*$ (and $p_1(t^*)<1$).
\item The normal game has two phase transitions: there
    exist $0<t^-<t^+<1$ such that $D$ increases
    continuously from $0$ to positive values at $t^-$,
    and jumps discontinuously from one positive value to
    another at $t^+$.
\item The escape game has a continuous phase transition:
    there exists $t_e\in(0,1)$ such that $\ei=0$ for
    $t\leq t_e$ while $\ei\in(0,1)$ for $t>t_e$, and
    $\ei$ is a continuous function of $t$.
\end{enumerate}
\end{prop}

Notwithstanding the above examples, the next result
establishes some general patterns concerning the nature of
phase transitions.  In particular,
for certain simple families of distributions,
phase transitions are indeed continuous
for the normal and mis\`ere games but discontinuous for the
escape game.  To make the statements precise, we need two
different metrics on offspring distributions
$\bp=(p_0,p_1,\dots)$.  Let $M_0$ be the space $\{\bp:
\sum_i p_i=1\}$ of all offspring distributions, with the
$\ell^1$ metric $d_0(\bp, \bq):=\sum_i|p_i-q_i|$. Let $M_1$
be the space $\{\bp\in M_0: \sum_i i p_i<\infty\}$ of
distributions with finite mean $\mu$, with the metric
$d_1(\bp, \bq):=\sum_i i|p_i-q_i|$.

\begin{thm}[Phase transitions] \label{cont}
Consider a Galton-Watson process with
offspring  distribution $\bp$.
\begin{enumerate}[\rm (i)]
\item The probabilities $N,P,\tN,\tP,\si,\sii$ are lower
    semicontinuous as functions of $\bp$ with respect to
    $d_0$.  Hence, $D$ is upper semicontinuous, and $N$ and $P$ are continuous on $\{\bp:D=0\}$; and
    similarly for the mis\`ere game.
\item The probabilities $D$ and $\tD$ are continuous with respect to
    $d_0$ on the set of distributions $\bp$ supported on $\{0,1,2\}$ and satisfying $0<p_0<1$.
\item The set $\{\bp: \ei>0\}$ contains $\{\bp\in M_1: \mu p_1>1\}$ and
    is closed with respect to $d_1$ in $\{\bp\in M_1: \mu p_1<1\}$. We
    have $\ei>0$ if and only if $\eii>0$.
\end{enumerate}
\end{thm}
Part (iii) above deserves some explanation.  The condition
$\mu p_1>1$ corresponds to a particularly simple
explanation for an Escaper win: there is a supercritical
branching process on which Escaper can always leave
Stopper with exactly one legal move.
(See Proposition \ref{p1muprop} and its proof in Section
\ref{sec:continuity} for more details.)
The result says
essentially that a continuous transition between $\ei>0$
and $\ei=0$ can occur only where the above criterion is the
sole explanation for escapes, i.e.\ when the transition
occurs as a result of crossing the boundary of the region
$\mu p_1>1$. Elsewhere, the escape region is closed and
thus includes its critical surface.

It would be desirable to find more general conditions under which the conclusion of part (ii) holds (although \cref{exotic} shows that it does not hold in full generality).  What is the largest $k$ for which it holds for all distributions with support $\{0,\ldots,k\}$?  Can it be established for some broader class of ``reasonable distributions"?


Finally, we investigate further the topology of the
region of distributions giving positive draw
probability, and the nature of the phase
transitions which can occur, by considering quantities
related to the length of the game.

Consider the normal or mis{\`e}re game. We
define the \df{length} of the game with optimal play, denoted by $T$, as follows.
Suppose that the game is a win for one of the players.
Then $T$ is the number of turns in the game (i.e.\ the distance from the root to the
leaf where the game ends) if the winning player tries to win
the game as quickly as possible while the losing player tries to
prolong it as much as possible. Equivalently, $T$ is
the smallest $n$ such that some player has a strategy that
ensures a win in $n$ turns or fewer.
(From a simple compactness result, Proposition \ref{compact} below,
such an $n$ exists if the game is not a draw.)
If the game is a draw with optimal play, define $T=\infty$.

Next, say that a path from the root to a vertex $v$ is
a \textbf{forcing path} if each player has a strategy
that guarantees that either they do not lose, or that
the game passes through $v$. Let $T^*$ be the supremum of the lengths of all forcing paths.
If the game is a draw,
then trivially the path to any vertex $v$ is forcing, since
both players have strategies that guarantee not to lose,
and $T^*=\infty$. On the other hand if one player has a
winning strategy, then $T^*$ is finite, and
we have the following interpretation:
although the other player is destined to lose eventually,
they can control the path of the game for the first $T^*$
turns, unless the opponent is willing to give up the win.
Note that $T^*\leq T$.


\begin{thm}[Length of the game]\label{thm:length}
Consider the normal or mis{\`e}re game on a Galton-Watson tree,
with offspring distribution $\bp$.
Write $\cB$ for the set of offspring distributions such
that the probability of a draw is $0$, and $\partial\cB$ for its boundary in $M_0$.
\begin{enumerate}[\rm (i)]
\item If $\bp$ is in the interior of $\cB$, then $\E T<\infty$ and $\E T^*<\infty$.
\item If $\bp\in\partial\cB\intersect\cB$, then $\E T=\infty$ and $\E T^*=\infty$.
\item
Along any sequence of offspring distributions
in $\cB$ converging in $M_0$ to a distribution
in $\partial\cB\intersect\cB$, we have $\E T\to\infty$ and $\E T^*\to\infty$.
\item
Along any sequence of offspring distributions
in $\cB$ converging in $M_0$ to a distribution
in $\partial\cB\setminus\cB$, we have $\E T\to\infty$.
\end{enumerate}
\end{thm}

The set $\partial\cB\intersect\cB$ is the set of distributions in
the boundary of $\cB$ which have draw probability $0$;
hence we may interpret $\partial\cB\intersect\cB$ as the
set of ``continuous phase transition" points, and similarly
the set $\partial\cB\setminus\cB$ as the set of ``discontinuous
phase transition" points.

In parts (iii) and (iv) of Theorem \ref{thm:length},
we see that $\E T$ blows up as
we approach the boundary of $\cB$, and that $\E T^*$ blows up if
we approach a continuous phase transition point.
It would be convenient to complete the result with the statement
that $\E T^*$ does \textit{not} blow up at a
discontinuous phase transition point.
However such a statement is not true without further qualification.
For the case of the normal game, let $x^*$ be the unique fixed point in $[0,1]$ of the function $F$.
During the proof of Theorem \ref{thm:length}, we show that
$\E T^*\to\infty$ precisely if at the limit point, $F'(x^*)=-1$.
At a continuous phase transition point (where new fixed points
of the function $F^2$ emerge smoothly from the fixed point $x^*$),
we will show that indeed $F'(x^*)=-1$. At discontinuous phase
transition points (where new fixed points of $F^2$ are created
away from $x^*$), it is not generally the case that $F'(x^*)=-1$.
However, it can occur that $F'(x^*)=-1$; we could loosely interpret such cases by saying that a continuous phase transition is occurring,
but it is masked by a simultaneously occurring discontinuous
phase transition. (For the case of the mis\`ere game, replace
the function $F$ by the function $H$ throughout.)

Accordingly, we conjecture that the correct
completion of the result in Theorem \ref{thm:length}(iii)-(iv) is as follows: $\E T^*$ stays bounded if the limit distribution
is in $\partial\cB\setminus(\partial(\partial\cB\intersect\cB))$ (a phase transition point which is separated from the set
of continuous phase transition points),
and $\E T^*\to\infty$ if the limit distribution is
any other point in $\partial\cB$. However, we do not
have a proof of this statement.




It is instructive to compare the phase transitions considered here
with those involving other properties of
branching processes.
First let $\cA$ be the set of offspring distributions for
which the branching process dies out with probability $1$
(i.e.\ the probability that an infinite path exists is $0$),
and let $\bp^*$ be the
degenerate distribution with $p^*_1=1$ and $p^*_k=0$ for $k\ne 1$.
Then $\cA\union\{\bp^*\}$ is closed as a subset of $M_0$
(it is well known that $\cA$ consists precisely of
those distributions with mean less than or equal to $1$, except
for $\bp^*$). Along a sequence of distributions in $\cA$ converging
to a point in $\partial\cA$, the expected length of the longest
path in the tree goes to $\infty$, and the probability of extinction
is continuous at the boundary of $\cA$ (except at $\bp^*$).

On the other hand, consider the event that the tree of the branching process contains a complete infinite binary tree, rooted at the root of the branching process. Let $\cC$ be the set of offspring distributions
for which this event has probability $0$. Now it is possible
to show that the set $\cC\intersect M_1$ is open as a subset of $M_1$. (We do not write the proof here, but we observe that a closely related property involving the $3$-core of sparse random graphs converging
locally to a branching process is studied extensively by Janson
\cite{Janson2009}). Hence within $M_1$, the phase transitions at
the boundary of $\cC$ are discontinuous; for example, it was
shown by Dekking \cite{Dekking} that for the particular case of a Poisson$(\lambda)$ offspring distribution, the probability of existence of such a binary subtree is $0$ for $\lambda<\lambda_c\approx 3.35$,
and jumps to around $0.535$ at $\lambda_c$.

In contrast to the previous two paragraphs, we see that
for the case of the draw probability,
the set $\cB$ considered in Theorem \ref{thm:length} is neither
open nor closed. Along a sequence of distributions
converging to a distribution in $\partial\cB\intersect\cB$,
we see a continuous phase transition as in the case of
survival/extinction of a branching process. Indeed,
in the proof we show that the union of all forcing paths
is itself a (two-type) Galton-Watson process which
is a subtree of the game tree, and
which itself approaches criticality (for survival/extinction) at the phase transition point.
In this case we can explain the emergence of draws by the divergence
to infinity of the length of a forcing path available to the
losing player. On the other hand, the case of a discontinuous
phase transition is much more similar to that observed for the
set $\cC$ defined in terms of the occurence of a binary tree within a branching process; here it seems that the emergence of draws cannot
be explained in terms of a single path, but intrinsically involves
a more complicated branching structure.

\subsection*{Background and related work}

Two recent articles by the current authors together with Basu and
W\"astlund \cite{undir} and with Marcovici \cite{dir} address these games and their variants on other structured random graphs.

Specifically, \cite{dir} considers the normal game and a variant of the mis\`ere game on percolation clusters of oriented Euclidean lattices.  Using probabilistic cellular automata and hard core models, it is proved that draws occur in dimensions $d\geq 3$ and greater (on certain lattices) if the percolation parameter is large enough, but not in dimension~$2$.  Many questions remain unresolved, such as monotonicity of the draw probability in the percolation parameter (which would imply uniqueness of the phase transition for $d\geq 3$).

On the other hand, \cite{undir} is concerned with percolation on \emph{unoriented} lattices.  The normal game as defined earlier is less interesting on an undirected graph, since (unless the starting
vertex has no neighbour) either player can draw by immediately reversing every move of the other player.  We therefore consider a different extension of the rules, in which the token is forbidden to ever revisit a vertex, giving a game that we call \emph{Trap}.  (On a tree, Trap and the normal game are clearly equivalent).  For percolation on Euclidean lattices in any dimension $d\geq 2$, it is unknown whether Trap has draws for some nontrivial percolation parameter.  Simulation evidence tends to support a negative answer in $d=2$, while analogy with the directed case might suggest a positive answer for $d\geq 3$.  The article \cite{undir} uses connections with maximal matchings and bootstrap percolation to establish finite scaling results on a biased percolation model where vertices have two different occupation parameters according to their parity, thus favoring one player.

Compared with the cases discussed above, the recursive structure of Galton-Watson trees allows a considerably deeper analysis.  Two special cases of the normal game have been partially analysed before: the phase transition for the Binomial$(2,p)$ offspring was found in the PhD thesis of one of the current authors \cite{phd}.
The case of the Poisson offspring family is closely related to
the analysis of the \emph{Karp-Sipser algorithm} used to find
large matchings or independent sets of a graph,
introduced by Karp and Sipser in
\cite{KarpSipser}. For the case of
\ER random graphs $G(n,\lambda/n)$ they identified
a phase transition at $\lambda=e$
corresponding to that noted in Theorem \ref{examples} above, and
dubbed it the ``$e$-phenomenon"; the link to
games is not described explicitly but the choice of notation and
terminology makes clear that the authors were aware of it.

We mention some recent papers particularly
closely related to the current study.
In \cite{MartinStasinski}, Martin and Stasi{\'n}ski
consider minimax recursions defined on Galton-Watson trees
with no leaves, truncated at some depth $k$. Terminal
values at the level-$k$ nodes are drawn independently from
some common distribution. Such recursions give the value of a
general class of two-player combinatorial games; the behaviour
of the value associated to the root is studied as $k\to\infty$.
Johnson, Podder and Skerman \cite{JohnsonPodderSkerman}
study a wider class of recursions on supercritical Galton-Watson trees, 
with a particular focus on cases where the one-level generating-function recursion has multiple fixed points. 
Broutin, Devroye and Fraiman \cite{BroutinDevroyeFraiman}
study related questions for minimax functions and more general
recursions, in the case of Galton-Watson trees conditioned to have a 
given number of vertices.

Other work on combinatorial games in random settings includes
the study of positional games (such as Maker-Breaker games)
on random graphs, for example \cite{maker-breaker-geometric, biased,
StojakovicSzabo2005},
and \cite{frogs} which deals with matching games played on random point sets, with an intimate connection to Gale-Shapley stable marriage.  In another direction, \cite{logic} uses certain games as tools for proving statements involving second-order logic on random trees, and \cite{wastlund} uses a game in the analysis of optimization problems in a random setting.  One striking observation from all these examples is that games, by their competitive nature, often automatically tease out and magnify some of the most interesting and subtle structural properties of random systems.


\section{Recursions and compactness}

In this section we give the basic recursions underlying
analysis of the games.  First consider the normal game on
any directed acyclic graph with vertex set $V$, and let
$\vn$ be the set of vertices $v$ for which the game is a
next-player win if the token is started at $v$. Similarly
define $\vp$ and $\vd$ to be the sets of vertices from
which the game is a previous-player win and a draw
respectively (so that $(\vn,\vp,\vd)$ is a partition of
$V$). In the case of the Galton-Watson tree with offspring
distribution $\bp$ we have $N=N(\bp)=\P(o\in\vn)$, etc. Let
$v$ be a vertex and let $\Gamma=\Gamma(v)$ be its
out-neighborhood, i.e.\ the set of end-vertices of the
edges leading from $v$.  By considering the first move, it
is immediate that the following relations hold.
\begin{equation}\label{recur-inf}
\begin{aligned}
v\in\vn \quad\text{ iff }\quad& \Gamma\cap\vp\neq\emptyset;\\
v\in\vp \quad\text{ iff }\quad& \Gamma\subseteq\vn;\\
v\in\vd \quad\text{ iff }\quad& \Gamma\cap\vp=\emptyset
\text{ but }\Gamma\cap\vd\neq\emptyset.
\end{aligned}
\end{equation}

Similar relations hold for the other games. However, these
relations are not in general sufficient to determine the
sets.  For example, consider the normal game on a singly
infinite path directed towards infinity. Clearly, every
vertex belongs to $\vd$, but two other possible solutions to
\eqref{recur-inf} assign vertices alternately to $\vp$ and
$\vn$ along the path.

The following refinement will enable us to choose the
correct solutions.  For $n\geq 0$, let $\vn_n$ be the set
of starting vertices from which the Next player has a
winning strategy that guarantees a win after strictly fewer
than $n$ moves (counting the moves of both players).
Similarly, let $\vp_n$ be the set of vertices from which
the Previous player can guarantee a win in fewer than $n$
moves.  In particular we have $\vn_0=\vp_0=\emptyset$.  Let
$\vd_n=V\setminus(\vn_n\cup\vp_n)$.  This may be
interpreted as the set of starting vertices from which the
game is drawn under the convention that we declare the game a draw whenever it lasts for $n$ moves.  By
considering the first move again, we have for $n\geq0$,
\begin{equation}\label{recur}
\begin{aligned}
v\in\vn_{n+1} \quad\text{ iff }\quad& \Gamma\cap\vp_n\neq\emptyset;\\
v\in\vp_{n+1} \quad\text{ iff }\quad& \Gamma\subseteq\vn_n.
\end{aligned}
\end{equation}
(It is easy to deduce that $\vn_1=\emptyset$, while $\vn_{2k}=\vn_{2k+1}$ and $\vp_{2k-1}=\vp_{2k}$.)

Similarly, let $\vnm,\vpm,\vdm$ be the sets of starting
vertices from which the mis\`ere game is a Next player win,
a Previous player win, and a draw respectively.  Let
$\vnm_n,\vpm_n$ be the sets from which the relevant player
can guarantee to win in fewer than $n$ moves, and let
$\vdm_n=V\setminus(\vnm_n\cup\vpm_n)$. Then we have
\begin{equation}\label{recur-mis}\begin{aligned}
 v\in\vnm_{n+1} \quad\text{ iff }\quad&
\Gamma\cap\vpm_n\neq\emptyset
 \text{ or }\Gamma= \emptyset;\\
v\in\vpm_{n+1} \quad\text{ iff }\quad& \Gamma\subseteq\vnm_n
\text{ and }\Gamma\neq \emptyset.
\end{aligned}\end{equation}

For the escape game, let $\vsi,\vsii$ be the sets from
which Stopper wins, when he has the first move and the
second move respectively, and let $\vei,\veii$ the sets
where Escaper wins, when moving first and second
respectively. Let $\vsi_n,\vsii_n$ be the sets from which
Stopper can win in fewer than $n$ moves, and let
$\vei_n=V\setminus \vsii_n$ and $\veii_n=V\setminus
\vsi_n$.  (These are Escaper's winning sets if we declare
Escaper the winner after the $n$th move).  We have
\begin{equation}\label{recur-esc}\begin{aligned}
v\in\vsi_{n+1} \quad\text{ iff }\quad& \Gamma\cap\vsii_n\neq\emptyset \text{ or }\Gamma=\emptyset;\\
v\in\vsii_{n+1} \quad\text{ iff }\quad& \Gamma\subseteq\vsi_n.
\end{aligned}\end{equation}

To use the above relations, we need the following simple
but important fact: if a player can win (or, in the escape
game, if Stopper can win), then they can guarantee to do
so within some finite number of moves which they can specify in advance.  This follows from
compactness arguments going back to \cite{konig}. For the reader's
convenience, we include a proof.

\begin{prop}[Compactness]\label{compact}
Let $\G$ be a directed acyclic graph with all out-degrees
finite.  We have $\vn=\bigcup_{n=0}^{\infty} \vn_n$, and
similarly for each of $\vp,\vnm,\vpm,\vsi,\vsii$.
\end{prop}

\begin{proof}
Consider first the normal game.  Let $\vn':=\vn\setminus
\bigcup_{n=0}^{\infty} \vn_n$ and $\vp':=\vp\setminus
\bigcup_{n=0}^{\infty} \vp_n$ be the sets from which the
relevant player can win, but cannot guarantee to do so
within any finite number of moves.  We must show that
$\vn'=\vp'=\emptyset$.

If $v\in \vn'$ then the out-neighbourhood $\Gamma(v)$
contains some vertex in $\vp'$ but none in
$\vp\setminus\vp'$ (otherwise the Next player could win in
finitely many moves).  If $v\in \vp'$ then all vertices of
$\Gamma(v)$ lie in $\vn$, and we claim that at least one of
them lies in $\vn'$. Indeed, if not then each
$w\in\Gamma(v)$ lies in $\vn_{m(w)}$ for some $m(w)$. But
then $M:=\max\{m(w):w\in\Gamma(v)\}$ is finite, and so
$v\in \vp_{M+1}$, a contradiction.

We now claim that from any vertex in $\vn'$, the Previous
player has a strategy that guarantees a draw or better.
Indeed, if the Next player is foolish enough to move to a
vertex in $\vn\cup\vd$ then the Previous player simply
plays to win or draw as usual. If the Next player instead
moves to a vertex in $\vp'$ then the Previous player
replies by moving again to a vertex in $\vn'$.  The same
strategy allows the Next player to draw from any vertex in
$\vp'$. Hence, there are no such vertices.

For the mis\`ere game, we can reduce to the normal game on
a modified graph: from each vertex of out-degree $0$ we add
a single outgoing edge to a new vertex of out-degree $0$.
We now appeal to the normal game case already proved.

For the escape game, we can reduce to the normal game on a
different modified graph.  Fix a starting vertex $u$ and
suppose that Stopper moves first.  First, split each vertex
$v$ into two copies $v_0$ and $v_1$ to indicate whether it
is reached after an even or odd number of moves. Let the
token start at $u_0$.  Split edge $(v,w)$ into two edges
$(v_0,w_1)$ and $(v_1,w_0)$. The resulting graph is
bipartite. Finally, for any $v$ with out-degree $0$, add an
outgoing edge from $v_0$. The case when Stopper moves
second is handled similarly, except that in the final step
we instead add the outgoing edge to $v_1$.
\end{proof}

The finite out-degree assumption in the last result is
needed.  For instance, if $\G$ is a tree consisting of
outgoing paths of every even length $2,4,6,\ldots$
emanating from a root $o$, then the Previous player wins,
but the Next player can make the game arbitrarily long.

\section{Generating functions and fixed points}

We next prove \cref{main}. From now on we specialize to the
case $\G=\T$, the Galton-Watson tree with offspring
distribution $\mathbf{p}=(p_0,p_1,\ldots)$.  Recall that we
write $N=N(\bp)=\P(o\in \vn)$, and similarly for
$P,D,\tN,\tP,\tD,\si,\sii,\ei,\eii$. Recall the sets
$\vn_n$, etc.\ defined in the previous section.  Define the
associated probabilities $N_n:=\P(o\in \vn_n)$, etc.

On a tree, these probabilities may be interpreted as follows. Let $\T_n$ be the finite
subgraph of $\T$ induced by the set of vertices of \df{depth}
(i.e.\ distance from $o$) at most $n$. Consider the normal
game played on $\T_n$, but \emph{declared to be a draw if
the token ever reaches depth $n$}.  The outcome of this
game may be computed by assigning all depth-$n$ vertices of
$\T_n$ to $\vd$, and then using the recurrence
\eqref{recur} to classify the other vertices. Then $N_n$ is
the probability that the Next player wins starting from
$o$, and similarly for $P_n$ and $D_n$. Similarly,
$\tN_n,\tP_n,\tD_n$ be the outcome probabilities for the
mis\`ere game on $\T_n$ where we declare a draw at depth
$n$.  For the escape game, declare vertices at depth $n$ to
be \emph{wins for the escaper}; then
$\si_n,\sii_n,\ei_n,\eii_n$ are the relevant outcome
probabilities.

\begin{cor}[Truncation and limits]\label{trunc}
For any offspring distribution $\bp$, with the above
notation, we have $N=\lim_{n\to\infty}N_n$, and similarly
for $P,D,\tN,\tP,\tD,\si,\sii,\ei,\eii$.
\end{cor}

\begin{proof}
By \cref{compact} we have $N_n\nearrow N$ as $n\to\infty$.
Similarly, $P_n\nearrow P$.  (In fact, since the first player can only win in an odd number of moves, we have $N_{2k}=N_{2k+1}$ for all integers $k\geq 0$, and similarly $P_{2k+1}=P_{2k+2}$.)  Since $D=1-P-N$ and
$D_n=1-P_n+N_n$, we have $D_n\searrow D$.  The same
argument works for the mis\`ere game.  Similarly, for the
escape game we get $S^{(j)}_n\nearrow S^{(j)}_n$ for
$j=1,2$, but $E^{(1)}=1-S^{(2)}$ and $E^{(2)}=1-S^{(1)}$.
\end{proof}

Recall that we define the generating function
$G(x)=G_{\mathbf{p}}(x):=p_0+p_1 x+p_2 x^2+\cdots$,
which is a continuous, increasing, convex function from $[0,1]$ to $[0,1]$.  Recall that we also define the functions
\[
F(x):=1-G(x);\qquad H(x):=1-G(x)+p_0,
\]
which are decreasing and concave.
%
%

\begin{proof}[Proof of \cref{main}]
First consider the normal game.  \cref{trunc} gives
$(\n,\p,\d)=\lim_{n\to\infty} (\n_n,\p_n,\d_n)$.  We apply
the recursion \eqref{recur} at the root $o$, noting that
there is an independent copy of $\T$ rooted at each child.
We obtain for $n\geq 0$
$$\p_{n+1}=G(\n_n); \qquad 1-\n_{n+1}=G(1-\p_n).$$
This implies that $N_{n+2}=F^2(N_n)$ and
$1-P_{n+2}=F^2(1-P_n)$. Note also that $\n_0=\p_0=0$.
Therefore, since $F^2$ is increasing and continuous,
\begin{align*}
\n&=\lim_{n\to\infty} F^{2n}(0)=\min\fp(F^2);\\
1-\p&=\lim_{n\to\infty} F^{2n}(1)=\max\fp(F^2).
\intertext{Hence,}
\d&=1-\n-\p=\max\fp(F^2)-\min\fp(F^2).
\end{align*}

The arguments for the other games are similar.  For the mis\`ere game, the recursion \eqref{recur-mis} gives
$\nm_{n+2}=H^2(\nm_n)$ and $1-\pm_{n+2}=H^2(1-\pm_n)$, so that
\begin{align*}
\nm&=\lim_{n\to\infty} H^{2n}(0)=\min\fp(H^2);\\
1-\pm&=\lim_{n\to\infty} H^{2n}(1)=\max\fp(H^2);\\
\dm&=1-\nm-\pm=\max\fp(H^2)-\min\fp(H^2).
\end{align*}

For the escape game, \eqref{recur-esc} gives
$\si_{n+1}=H(\ei_n)$ and $\ei_{n+1}=F(\si_n)$, so that
\begin{align*}
\si&=\lim_{n\to\infty} (H\circ F)^{n}(0)=\min\fp(H\circ F);\\
\ei&=\lim_{n\to\infty} (F\circ H)^{n}(1)=\max\fp(F\circ H);\\
\sii&=1-\ei;\qquad
\eii=1-\si.\qedhere
\end{align*}
\end{proof}

We note a sense in which the escape game is intermediate between the other two games: its outcome probabilities arise from \emph{alternately} iterating the two functions $F$ and $H$ that govern the others.
For later use we note the following relations between
outcome probabilities of the games on the full tree.
\begin{cor}\label{extras}
For any offspring distribution we have:
\begin{gather*}
1-\p=F(\n); \qquad \n=F(1-\p);\\
1-\pm=H(\nm); \qquad \nm=H(1-\pm);\\
\si=H(\ei); \qquad \ei=F(\si).
\end{gather*}
\end{cor}

\begin{proof}
These can be deduced either by taking limits as
$n\to\infty$ of the corresponding recurrences in the above
proof, or by directly applying \eqref{recur-inf} and its
analogues for the other games.
\end{proof}

\section{Inequalities}

In this section we prove the inequalities of \cref{ineq}.
The fact that no other inequalities hold in general is
proved in \cref{sec-examples}.

\begin{proof}[Proof of \cref{ineq} (i)]
As remarked earlier, these inequalities of probabilities reflect inclusions that hold more generally.  Specifically, for the games on any directed acyclic graph $\G$, we have
$$\vn\subseteq\vsi; \quad \vnm\subseteq\vsi;\quad \vp\subseteq\vsii; \quad
\vpm\subseteq\vsii.$$ Indeed, the starting vertex lies in
$\vn$ if and only if the first player can ensure that the
game reaches a vertex of out-degree zero after an odd
number of steps. And the vertex lies in $\vnm$ if and only
if the first player can ensure that the game reaches a
vertex of out-degree zero after an even number of steps. In
either case, Stopper (if playing first) can win the escape
game by using the same strategy.  This gives the first two
inclusions. Similarly, considering the second player gives
the last two inclusions.
\end{proof}

\begin{proof}[Proof of \cref{ineq} (ii)]
We show that $\sii\leq \si$ and $\pm\leq\nm$ for any
Galton-Watson tree $\T$. Consider the escape game, and
suppose Stopper has first move. We propose a partial
strategy for Stopper.  If the root has no children, Stopper
wins immediately. If the root has one or more children, let
Stopper move to a child chosen uniformly at random (without
looking at the remainder of the tree). The rest of the game
is then played in a subtree with the same law as $\T$, with
Stopper moving second. This yields
\begin{align*}
\si&\geq p_0+(1-p_0)\sii \geq \sii.
\end{align*}
In the mis\`ere game, the first player can follow the same
strategy, to give
\begin{align*}
\nm&\geq p_0+(1-p_0)\pm\geq \pm.\qedhere
\end{align*}
\end{proof}

Moving on to the more interesting inequalities in
\cref{ineq} (iii), we start with some lemmas.

\begin{lemma}\label{squarelemma}
Consider any offspring distribution.  We have $H'(x)\geq
-1$ for all $x\leq \nm$. If $1-\pm>\nm$ (i.e.\ if $\dm>0$)
then $H'(x)\leq -1$ for all $x\geq 1-\pm$.
\end{lemma}

\begin{proof}
If $p_0\in\{0,1\}$ then the lemma is easy to check.
Therefore assume that $0<p_0<1$.  Since $H$ is concave, it
is enough to check the values of $H'$ at $\nm$ and $1-\pm$.
Recall from the proof of \cref{main} that $\nm$ is the
smallest fixed point of $H^2$ in $[0,1]$, and $1-\pm$ is
the largest fixed point. Recall also that
$\lim_{n\to\infty} H^{2n}(0)=\nm$.  We claim that the
sequence $(H^{2n}(0))_{n\geq 0}$ is strictly increasing.
Indeed, we have $H^2(0)>0$, and we can apply the strictly
increasing function $H^2$ repeatedly to both sides.

Suppose first that $H^2$ has only one fixed point.  Then
$H$ has the same fixed point, i.e.\ $H(\nm)=\nm$. Suppose
for a contradiction that $H'(\nm)<-1$. The idea is that
$\nm$ is an unstable fixed point for $H$ under iteration.
More precisely, since $H$ is continuous and concave, we
have for some $\epsilon>0$ that $H'(x)<-1$ for all
$x>\nm-\epsilon$. Since $H^{2n}(0)$ is strictly increasing
with limit $\nm$, we have $\nm-\epsilon<H^{2m}(0)<\nm$ for
some $m$.  But then the assumption on $H'$ gives that the
next two iterations move the iterate further from $\nm$,
i.e.
$$\nm-H^{2m+2}(0)>H^{2m+1}(0)-\nm>\nm-H^{2m}(0),$$
contradicting that $H^{2n}(0)$ is increasing.

On the other hand, if $H^2$ has more than one fixed point,
then $\nm$ and $1-\pm$ are the smaller and larger points of
a two-cycle of $H$, with $H(\nm)=1-\pm$ and $H(1-\pm)=\nm$.
Now consider the square $[\nm, 1-\pm]^2$. The graph of the
function $H$ passes through the top-left and bottom-right
corners of this square. Since $H$ is concave, it follows
that $H'(1-\pm)\leq -1$ and $H'(\nm)\geq -1$ as required.
\end{proof}

\begin{lemma}\label{pnlemma}
For any offspring distribution, $\pm_n\leq\p_n$ for all $n\geq0$.
\end{lemma}

\begin{proof}
The result is true for $n=0,1$, since $\p_0=\pm_0=\pm_1=0$
and $\p_1=p_0$. So it will be enough to show that $\pm_n\leq \p_n$
implies $\pm_{n+2}\leq \p_{n+2}$.

So suppose that $\pm_n\leq \p_n$. Then, using the
recurrences in the proof of \cref{main}, and the fact that
$F^2$ is increasing,
\begin{align*}
\p_{n+2}-\pm_{n+2}
&=(1-\pm_{n+2})-(1-\p_{n+2})
\\
&=H^2(1-\pm_{n})-F^2(1-\p_{n})
\\
&\geq H^2(1-\pm_{n})-F^2(1-\pm_{n})
\\
&= H[H(1-\pm_n)]-H[H(1-\pm_n)-p_0]+p_0
\\
\intertext{Since $H$ is concave, the last expression is at least}
&p_0 H'[H(1-\pm_n)]+p_0.
\end{align*}
Now $1-\pm_n \searrow 1-\pm$, and so $H(1-\pm_n)\nearrow H(1-\pm)=\nm$. In
particular $H(1-\pm_n)\leq \nm$, and so by Lemma \ref{squarelemma},
$H'(H(1-\pm_n))\geq -1$. Hence $\p_{n+2}-\pm_{n+2}\geq 0$ as required.
\end{proof}

\begin{proof}[Proof of \cref{ineq} (iii)]
The inequality $\pm\leq \p$ follows immediately from
\cref{pnlemma,trunc}.

For the inequality $\d\leq \dm$ it will similarly be enough
to prove that $\d_n\leq \dm_n$ for all $n$.  Again we
proceed by induction. We have $\d_0=\dm_0=1$. Suppose that
$\d_{n}\leq\dm_{n}$. From Lemma \ref{pnlemma} we have
$\pm_{n}\leq\p_{n}$. Then, since $F$ is decreasing and
concave, and $F$ and $H$ differ by a constant, and using
the recurrences from the proof of \cref{main},
\begin{align*}
\d_{n+1}
&=1-P_{n+1}-N_{n+1}\\
&=F(N_{n})-F(1-P_{n})\\
&=F(1-P_{n}-D_{n})-F(1-P_{n})\\
&\leq F(1-\pm_{n}-\dm_{n})-F(1-\pm_{n})\\
&=H(1-\pm_{n}-\dm_{n})-H(1-\pm_{n})\\
&=\dm_{n+1},
\end{align*}
completing the induction.

Finally we will show that $\pm\leq\n$ by considering two cases.
First suppose that $\dm>0$. Then by Lemma
\ref{squarelemma}, we have $H'(x)\leq -1$ for all $x\geq
1-\pm$.  Since $F$ and $H$ differ by a constant, also
$F'(x)\leq -1$ for all $x\geq 1-\pm$. Since $F(1)=0$, it
follows that
\begin{equation}\label{use}
F(1-\pm)\geq\pm.
\end{equation}
As proved above, we have $\pm\leq\p$. Since $F$ is
decreasing, this gives $F(1-\pm)\leq F(1-\p)=\n$. Combining
this with \eqref{use} gives $\pm\leq\n$ as required.

Now suppose instead that $\dm=0$.  Since $\d\leq \dm$ we have
also $\d=0$. Then $\n=1-\p$ is a fixed point of $F$, and
$\nm=1-\pm$ is a fixed point of $H$. Since $\pm\leq\p$ from
above, we have $\n\leq1-\pm$.
The functions $F$ and $H$ differ by a constant, and
both are concave and decreasing, so
\begin{equation*}\label{derivs}
H'(y)\leq F'(x)\leq 0 \quad\text{for all } x\in(0,\n)\text{ and } y\in(1-\pm,1).
\end{equation*}
Comparing the lengths of the intervals $(0,\n)$ and
$(1-\pm,1)$, this implies that
\begin{equation}\label{eitheror}
\pm\leq\n \quad\text{or}\quad H(1)-H(1-\pm)\leq F(\n)-F(0).
\end{equation}
In the former case we are done.  For the latter case note
that
\begin{align*}
H(1)-H(1-\pm)&=p_0-(1-\pm)=\pm+p_0-1;\\
F(\n)-F(0)&=\n-(1-p_0)=\n+p_0-1.
\end{align*}
Substituting into \eqref{eitheror} gives $\pm\leq \n$ in
the latter case also.
\end{proof}

\section{Examples}
\label{sec-examples}

In this section we use \cref{main} to prove
\cref{examples,exotic}, and to complete the proof of
\cref{ineq} by showing that no further inequalities hold.

\begin{proof}[Proof of \cref{examples} (i) -- binary
branching] Recall that $(p_0,p_1,p_2)=(1-p,0,p)$, so each individual has
either $0$ or $2$ children.  It turns out that in this example all relevant
quantities can be computed explicitly.  We have $G(x)=1-p+px^2$,
$F(x)=p(1-x^2)$ and $H(x)=1-px^2$. We treat the three games separately.

\medskip\paragraph{\em Normal Game}  \cref{main} gives the draw probability
$D$ in terms of the
fixed points of $F^2$, i.e.\ the zeros of $F^2(x)-x$. See \cref{binarynormalfigure} for graphs of this function.  We have the
factorization into two quadratics
$$F^2(x)-x=(p-x-px^2)(1-p^2-px+p^2x),$$
where the first factor equals $F(x)-x$.  Viewed as a
function of $x$, the first factor has exactly one zero, at
$x_0$ say, in $[0,1]$ for all $p\in[0,1]$. The second
factor has two distinct zeros $x_-<x_+$ in $[0,1]$ if and
only if its discriminant $p^2(4p^2-3)$ is positive, i.e.\
when $p>p_d:=\sqrt 3/2$.  Moreover, we have $x_-<x_0<x_+$
for $p>p_d$, while at $p=p_d$, all three roots coincide,
and the function has a stationary point of inflection on
the axis. (These last facts can be seen without further
computation: if $x_-$ is a fixed point of $F^2$ satisfying
$x_-<x_0$, then $x_+:=F(x_-)$ is also a fixed point, and
since $F$ is strictly decreasing we have $x_0<x_+$.
Moreover, the roots of a quadratic vary continuously with
its coefficients.) Therefore, by \cref{main} we have $D=0$
for $p\leq p_d$, and $D=x_+-x_-$ for $p>p_d$, giving the
claimed continuous phase transition.

\medskip\paragraph{\em Mis\`ere game} The analysis is similar.  We have the factorization
$$H^2(x)-x=(1-x-p x^2)(1-p-px-p^2x^2),$$
where the first factor is $H(x)-x$.  The first factor has exactly one zero at
$x_0$ say, and the second factor has two further zeros at $x_-<x_0<x_+$ if
and only if $p>p_m:=3/4$. For the same reasons as before, the transition is
continuous.

\medskip\paragraph{\em Escape Game} \cref{main} gives $\si=\min\fp(H\circ F)$.  See \cref{binaryescapefigure}.  We have
$$H\circ F(x)-x=(1-x)\bigl[1-p^3(1-x)(1+x)^2\bigr].$$
There is always a zero at $x=1$.  On $[0,1]$, the function $(1-x)(1+x)^2$ has
maximum $32/27$ at $x=1/3$. Therefore, there are two additional zeros if
$p^3\,32/27 \leq 1$, i.e.\ if $p\geq p_e:=3/2^5$.  The two additional zeros
are strictly less than $1$, and coincide at $x=1/3$ when $p=p_e$. Thus, $\si$
equals $1$ for $p<p_e$, and jumps to $1/3$ at $p=p_e$, giving the claimed
behaviour for $\eii=1-\si$.  \cref{extras} gives that $\ei>0$ if and only if
$\eii>0$.
\end{proof}
\begin{figure}
\begin{center}
\includegraphics[width=0.68\textwidth]{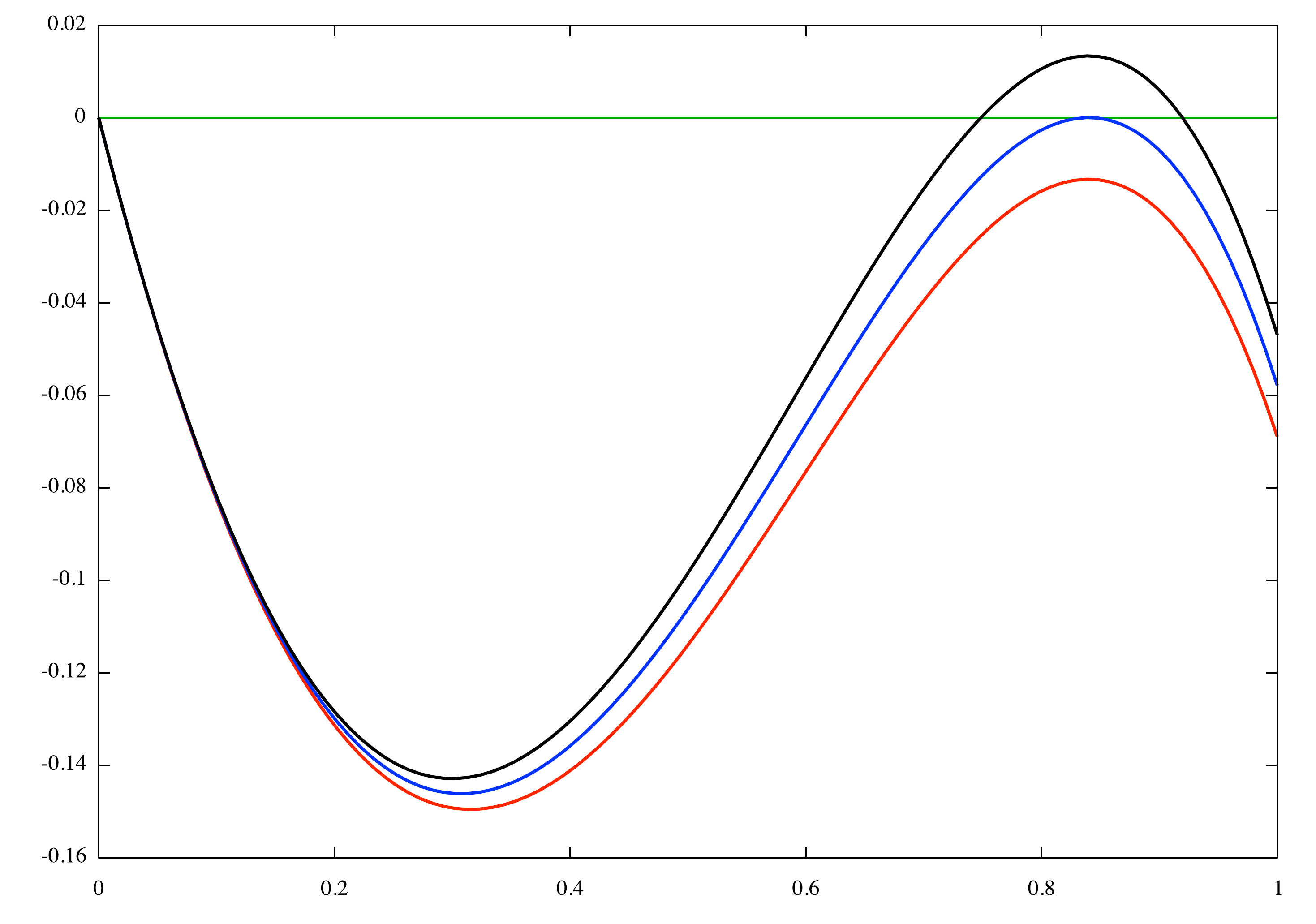}
\end{center}
\caption{
The function $F( H(x))-x$ for the binary branching distribution,
for the three values $p=0.935, 0.945, 0.955$.
The probability of escape is the largest root.
The lowest curve has its only root at 0. At the critical point (middle curve) a new root appears
and the probability of escape jumps to a positive value. Above the critical point (upper curve)
the function has three roots.
}\label{binaryescapefigure}
\end{figure}

%
%
%
%

%
%
%
%
%

\begin{proof}[Proof of \cref{examples} (ii) -- Poisson]
\sloppypar The offspring distribution is Poisson($\lambda$).  Thus, we have
$G(x)=e^{-\lambda(1-x)}$, $F(x)=1-e^{-\lambda(1-x)}$ and
$H(x)=1+e^{-\lambda}(1-e^{-\lambda x})$.  We will find that the behaviour of
the three games is qualitatively identical to that in the binary branching
case considered above, but that not all quantities can be computed explicitly.

\medskip
\paragraph{\em Normal game}
By \cref{main} we are interested in the fixed points of $F^2$.
Differentiating $F^2(x)-x$ twice with respect to $x$, we find that its first
derivative has exactly one turning point, a maximum at $x^*:=1-(\log
\lambda)/\lambda$, at which the first derivative equals $\lambda/e-1$.  We
deduce that when $\lambda\leq e$ the function $F^2(x)-x$ is strictly
decreasing on $[0,1]$, and thus has exactly one zero in $[0,1]$.

When $\lambda>e$, the function $F^2(x)-x$ has two turning points, a local
minimum followed by a local maximum.  Therefore it has at most three zeros.
We claim that it has exactly three.  To check this, note first that $F$
itself always has exactly one fixed point in $[0,1]$, say $x_0$, which
satisfies $\lambda=[-\log(1-x_0)]/(1-x_0)$.  This $x_0$ is also a fixed point
of $F^2$.  To show that $F^2(x)-x$ has three zeros it suffices to show that
its derivative is positive at $x_0$, which is equivalent to showing
$|F'(x_0)|>1$.  But $F'$ is negative and strictly decreasing in $x$, and
equals $-1$ precisely at $x=x^*$ (as defined above). Now $x_0$ is strictly increasing as a function of $\lambda$,
while $x^*$ is strictly decreasing as a function of $\lambda$. Therefore,
they coincide at exactly one $\lambda$, which is easily checked to be
$\lambda=e$. Therefore, we have $F'(x_0)<-1$ if and only if $\lambda>e$, as
required.

At the critical point $\lambda=e$, the function $F^2(x)-x$ has a stationary
point of inflection on the axis at $x_0$.  By \cref{main}, $D$ is the
distance between the zeros, which is continuous in $\lambda$, and equals $0$
if and only if $\lambda\leq e$.

\medskip\paragraph{\em Mis\`ere game}
The analysis and behaviour are similar to the normal game, except that the
critical point now has no closed-form expression. The derivative of
$H^2(x)-x$ has its maximum at $x^*=1-(\log \lambda)/\lambda$, at which the
first derivative is $\lambda e^{-1+\lambda e^{-\lambda}}-1$. This is positive
if and only if $\lambda>\lambda_m$, where $\lambda_m=2.103\ldots$ is the
solution of $\log\lambda+\lambda e^{-\lambda}=1$.  Thus, the function
$H^2(x)-x$ has one zero for $\lambda\leq \lambda_m$. Again, $H'(x^*)=-1$ for
all $\lambda$, and $x^*$ is increasing in $\lambda$, while the fixed point
$x_0$ of $F$ is decreasing (by implicit differentiation), with $x_0=x^*$ at
$\lambda=\lambda_m$.  By the same argument as before, this gives that
$H'(x^*)>-1$ if and only if $\lambda>\lambda_m$, hence $H^2$ has three fixed
points if and only if $\lambda=\lambda_m$.  And as before, at the critical
point $\lambda_m$ the function $H^2(x)-x$ has a stationary point of
inflection on the axis.  We deduce from \cref{main} that $\dm$ is continuous
in $\lambda$, and equals $0$ is and only if $\lambda\leq \lambda_m$.

\medskip\paragraph{\em Escape game}
The proof of \cref{main} gives that $\si$ is the minimum fixed point of
$H\circ F$.  This function always has a fixed point at $x=1$. Since $H\circ
F(x)= F^2(x)+e^{-\lambda}$, we can make use of the previous analysis of
$F^2$. For $\lambda<e$ the function $H\circ F(x)-x$ is decreasing and
therefore has exactly one zero. At $\lambda=e$ a stationary point of
inflection appears, but now it is strictly above the axis.  For all
$\lambda>e$ the function has a local minimum followed by a local maximum. For
$\lambda$ sufficiently close to $e$, the value of the function at its local
minimum is strictly positive.  But for $\lambda$ sufficiently large, it is
easy to check that the value at the local minimum is negative, and so $H\circ
F(x)-x$ has three zeros.  Moreover, we claim that the value of the function
at the local minimum is strictly decreasing as a function of $\lambda(>e)$,
so that it is negative if and only if $\lambda>\lambda_e$ for some critical
point $\lambda_e(>e)$.  To check this, it suffices to show that the function
$H\circ F(x)-x$ never has derivative zero with respect to $x$ and $\lambda$
simultaneously.  In fact, some algebra shows that the difference between the
two derivatives is never zero.  Finally, observe that $H\circ F(x)-x$ is
decreasing in a neighbourhood of $1$, so the locations of other zeros are
bounded away from $1$.  Hence $\eii=1-\si$ undergoes a discontinuous phase
transition at $\lambda=\lambda_e$ from $0$ to a positive value, and is
positive at the critical point, and is continuous elsewhere. Numerically, we
find $\lambda_e\approx 3.3185$. \cref{extras} shows that $\ei>0$ if and only
if $\eii>0$.
\end{proof}

\begin{proof}[Proof of \cref{examples} (ii) -- Geometric]
Let $\alpha\in(0,1)$ and let $p_k=(1-\alpha)\alpha^k$ for $k\geq 0$. Then
$G(x)=(1-\alpha)/(1-x\alpha)$.  It is straightforward to show that the
functions $F^2(x)-x$, $H^2(x)-x$, and $H\circ F(x)-x$ are all strictly
decreasing on $(0,1)$.  Therefore, by \cref{main}, the probabilities of draws
and escapes are zero.
\end{proof}



We now turn to the exotic examples of \cref{exotic}.
\begin{figure}
\begin{center}
\includegraphics[width=0.49\textwidth]{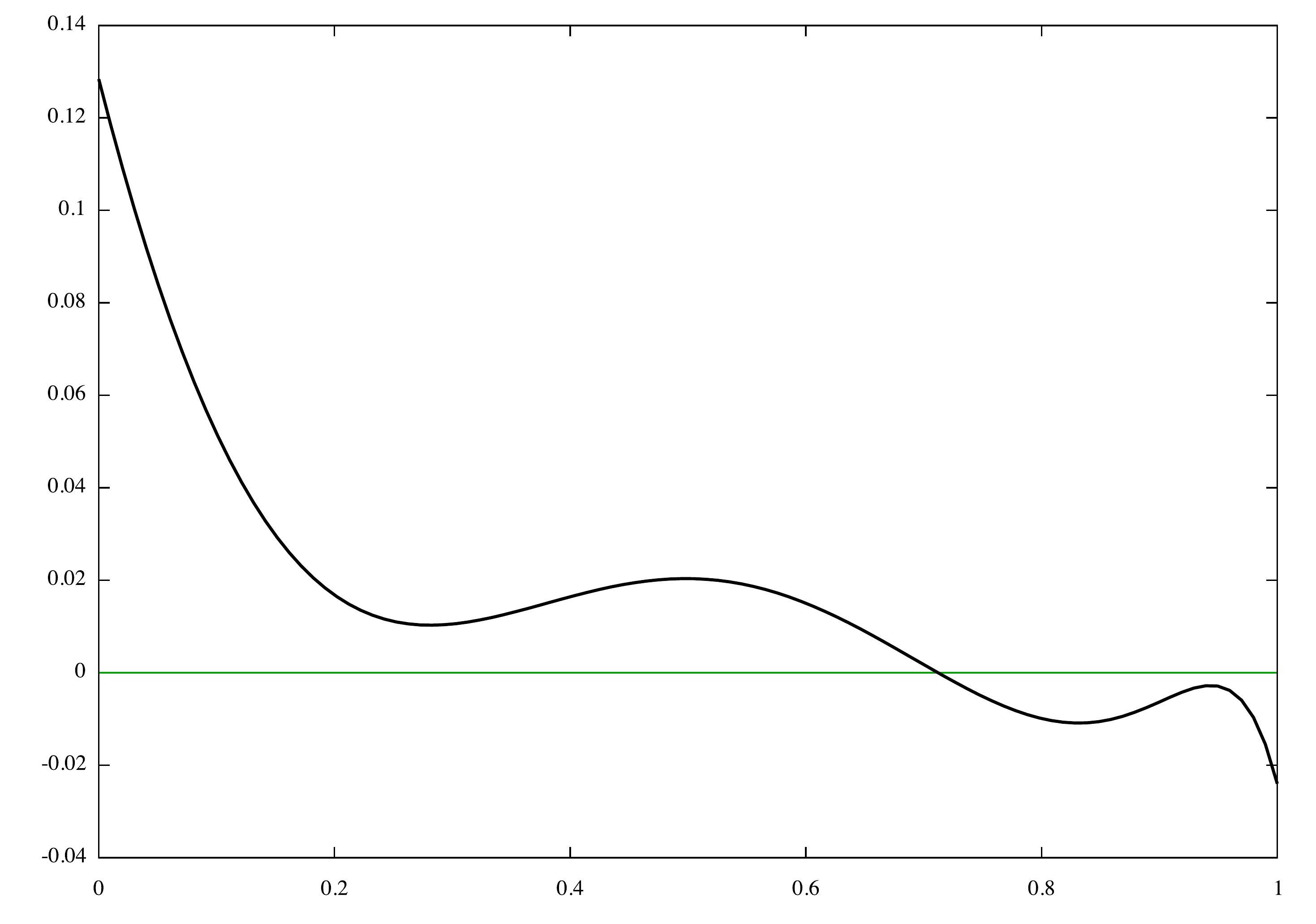}
\includegraphics[width=0.49\textwidth]{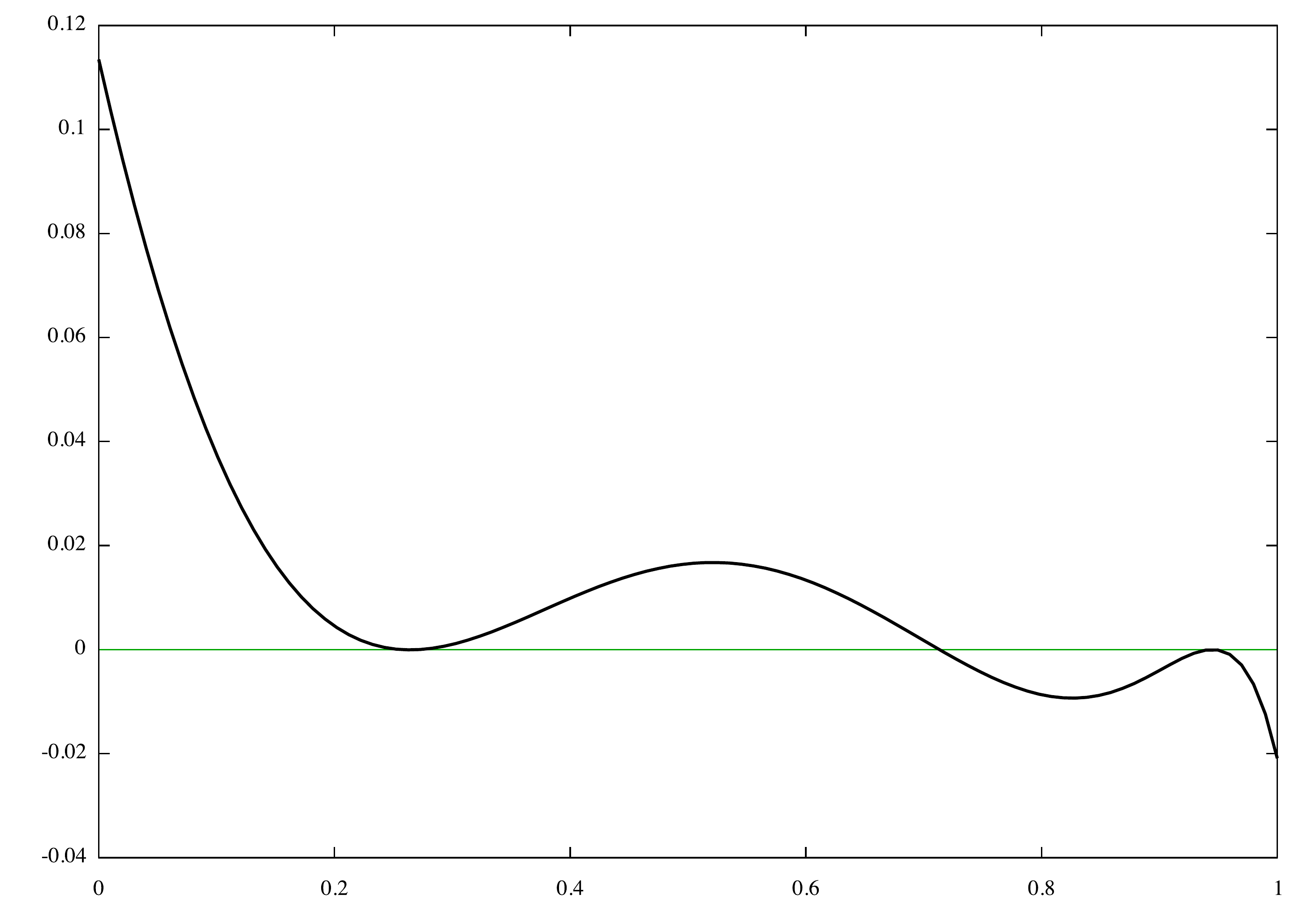}
\end{center}
\caption{
The function $F(F(x))-x$ for the family of distributions for \cref{exotic} (i), with $p=0.976$ on the left and $p=0.9791$ on the right.
On the left the function has a unique root and the probability of a draw is 0.
At the critical point, two new (double) roots appear and the probability of a draw
jumps to a positive value. Above the critical point the function has five roots.
}\label{discontinuousnormalfigure}
\end{figure}
\begin{proof}[Proof of \cref{exotic} (i)]
Let
\[
G(x)=(1-t)+t(0.5x^2+0.5x^{10});
\]
see \cref{discontinuousnormalfigure}.  There is a discontinuous phase transition at $t_c\approx
0.9791$.

For $t<t_c$, the equation $F^2(x)=x$ has a single solution.
For $t$ just smaller than $t_c$, we have $N=1-P\approx
0.7133$, and $D=0$.

At $t=t_c$, new solutions to $F^2(x)=x$ appear, at
$x^-\approx 0.264$ and $x^+\approx 0.945$. So the
probability of a draw jumps from 0 to $x^+-x^-\approx
0.681$. At $t_c$ itself the equation has three solutions,
with those at $x_-$ and $x_+$ being repeated roots, while
above $t_c$ the equation has five solutions.
\end{proof}

We remark that it is even possible for the draw probability
$D$ to jump from $0$ to $1$ as shown by the example $(p_0,p_1,p_2,p_3)=(\epsilon,\tfrac23,0,\tfrac13-\epsilon)$ discussed at the end of the final proof in this section.

\begin{figure}
\begin{center}
\includegraphics[width=0.69\textwidth]{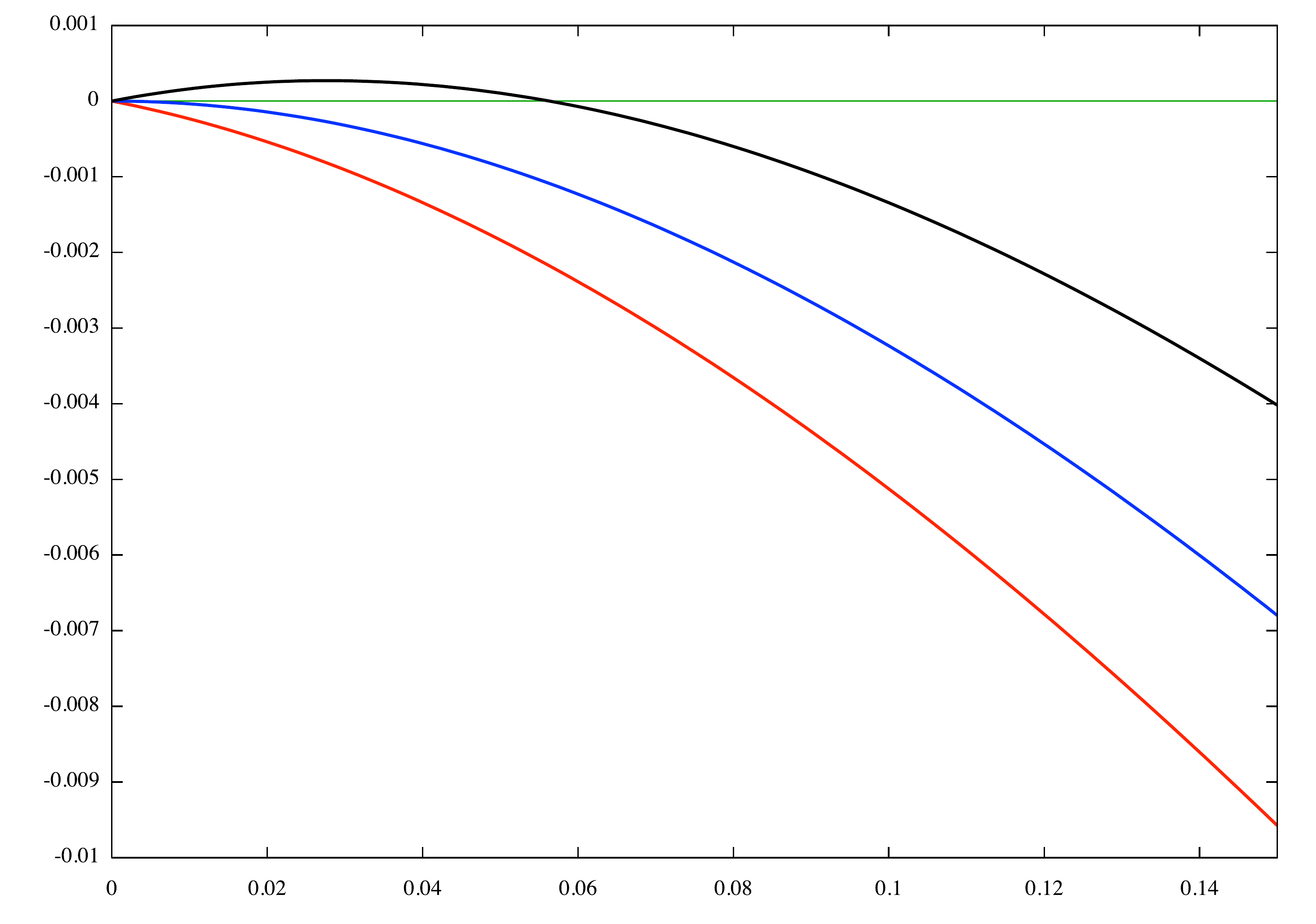}
\end{center}
\caption{
The function $F(H(x))-x$ for the family of distributions for \cref{exotic} (iii), with $\epsilon=-0.01$, $0$ and $0.01$.
For negative $\epsilon$ the function has only root, at 0, and the same is true
at the critical point $\epsilon=0$. As $\epsilon$ becomes positive
a new root emerges continuously from 0. }\label{cont-escape-figure}
\end{figure}
\begin{proof}[Proof of \cref{exotic} (ii)]
Let
\[
G(x)=(1-t)+t(0.15x+0.85x^{20});
\]
see \cref{cont-escape-figure}.
There are two phase transition points $t^-\approx 0.9877$
and $t^+\approx 0.99219$. When $t\leq t^-$, the equation
$F^2(x)=x$ has a single solution, and there are no draws.
At $t=t^-$ we see a continuous phase transition into a
region where the equation has three solutions and draws
occur; on $[t^-, t^+)$ the probability of a draw increases
continuously. Just below $t^+$ we have $N\approx 0.774$,
$P\approx 0.149$, $D\approx 0.077$.

At $t^+$ there is a discontinuous phase transition, and for
$t=t^+$ we have $N\approx 0.285$, $P\approx 0.020$,
$D\approx 0.695$. For $t>t^+$ there are seven solutions to
$F^2(x)=x$.
\end{proof}

\begin{proof}[Proof of \cref{exotic} (iii)]
Let
$$G(x)= \bigl(\tfrac1{18}-\epsilon\bigr)+\tfrac23
x+\bigl(\tfrac5{18}+\epsilon\bigr)x^3.$$
Note that
$p_1\mu>1$ if and only if $\epsilon>0$. At $\epsilon=0$,
the probability of escape is 0, but Proposition
\ref{p1muprop} tells us that for $\epsilon>0$ the
probability of escape must be positive. The function
$F(H(x))-x$ has $x=0$ as its only root for $\epsilon\leq0$,
but as $\epsilon$ becomes positive, the derivative of
$F(H(x))-x$ at $x=0$ moves from negative to positive, and a
second root emerges continuously from $0$. That is, $\ei>0$
for all $\epsilon>0$, with $\ei\to0$ as $\epsilon\to0$.
\end{proof}

The inequalities in \cref{ineq} will be proved in the next
section.  We conclude this section by giving examples
showing that no other inequalities hold in general.

\begin{proof}[Proof of Theorem \ref{ineq}, counterexamples]
We will give examples that rule out any inequality not listed in or
implied by Theorem \ref{ineq} (i)--(iii).

We start with a pair of trivial cases: if $p_0=1$ then
\begin{align*}
\p=\nm=\si=\sii&=1
>0=\n=\d=\pm=\dm=\ei=\eii,
\end{align*}
while if $p_0=0$ then
\begin{align*}
\d=\dm=\ei=\eii&=1
>0=\n=\p=\nm=\pm=\si=\sii.
\end{align*}


Another useful case is given by
$p_0=p_1=1/K$ and $p_{K^3}=1-2/K$, where $K$ is a large integer.
The following events hold with high probability as $K\to\infty$:
the root has $K^3$ children; at least one child of the root is a leaf;
at least one child of the root has exactly one child, which is a leaf.

As a result, the Next player wins both the normal and the mis\`ere games, and
Stopper wins the escape game when playing first, with high probability. Also,
since $p_1\mu\to \infty$, Escaper can win with high probability when playing
first, by arranging that Stopper never has any choice.  So we obtain that as
$K\to\infty$,
\begin{align*}
\n, \nm, \ei, \si&\to 1\\
\p, \d, \pm, \dm, \eii, \sii&\to 0.
\end{align*}

Next, in the case of binary branching in \cref{main} (i) with $p=p_2$ between
$\sqrt{3}/2=0.866\dots$ and $3/2^{5/3}=0.945\dots$, we have $\d>0$ while
$\ei=0$, so that $\d>\ei$ is possible.  An extreme case of the same example,
where we take $p=1-\epsilon$ with $\epsilon\to 0$, gives
\begin{align*}
\n&=2\epsilon+O(\epsilon^2)\\
\sii&=\epsilon+9\epsilon^2+O(\epsilon^3)\\
\p&=\epsilon+4\epsilon^2+O(\epsilon^3)\\
\nm&=\epsilon+O(\epsilon^3),
\end{align*}
so that $\n>\sii>\p>\nm$ is possible.

Finally, consider the example
$(p_0,p_1,p_2,p_3)=(\epsilon,\tfrac23,0,\tfrac13-\epsilon)$.  For
$\epsilon>0$ this has $D=\dm=0$, but for sufficiently small $\epsilon>0$ we
have $p_1\mu>1$, and therefore $\eii>0$ by \cref{cont}.  Thus $\eii>\dm$.  As an aside, we note that $D$ and $\dm$ both jump discontinuously to $1$ at $\epsilon=0$, because the tree has no leaves so the games cannot end).


It is straightforward to check that these examples show
that any inequality not ruled out by \cref{ineq} (i)--(iii)
may occur.
\end{proof}

\section{Continuity}
\label{sec:continuity}
In this section we prove \cref{cont}.

\begin{proof}[Proof of \cref{cont} (i)]
Recall that $N=N(\bp)$ is the increasing limit as $n\to\infty$ of $N_{2n}=H^{2n}(0)$.  But the latter is a continuous function of $\bp$ with respect to $d_0$ for each $n$.  Therefore $N$ is a lower semicontinuous function of $\bp$.  The same argument gives lower semicontinuity of $P,\nm,\pm,\si,\sii$.  Then $N+P$ is also lower semicontinuous, so $D=1-N-P$ is upper semicontinuous.  On $\{\bp:D=0\}$ we have $N=1-P$, so $N$ is upper and lower semicontinuous, hence continuous.  The same arguments apply to the mis\`ere game.
\end{proof}

The following simple observations will be useful for the proof of part (ii).

\begin{lemma}[Roots in pairs]\label{roots-in-pairs}
Let $\bp$ be any offspring distribution with $0<p_0<1$.  There is a unique fixed point $x^*$ of $F$ in $[0,1]$.  Besides $x^*$, all other fixed points of $F^2$ in $[0,1]$ can be partitioned into pairs of the form $\{x,F(x)\}$.  If $\bp$ is finitely supported (so that $F^2$ is a polynomial) and one element of such a pair is a \emph{repeated} root of $F^2(x)-x$, then so is the other.
\end{lemma}

\begin{proof}
First note that $F(x)-x$ is positive at $0$, negative at $1$, and strictly decreasing on $[0,1]$, so $F$ has a unique fixed point $x^*$ in $[0,1]$.  Clearly $x^*$ is also a fixed point of $F^2$.  If $x$ is any fixed point of $F^2$ then so is $F(x)$, and if $x\neq x^*$ then $F(x)\neq F(x^*)=x^*$.  Moreover if $x\in[0,1]$ then $F(x)\in[0,1]$.  Finally, the derivative of $F^2(x)-x$ is $F'(F(x))F'(x)-1=:\Delta(x)$, say.  If $x$ is a repeated root of $F^2(x)-x$ then $\Delta(x)=0$, but this implies that $\Delta(F(x))=F'(x)F'(F(x))-1=0$ also.
\end{proof}

\begin{proof}[Proof of \cref{cont} (ii)]
We prove continuity of $D$; the proof for $\dm$ is essentially identical.
Let $\cQ:=\{\bp=(p_0,p_1,p_2): \sum_i p_i=1,\ p_0\in(0,1)\}$ be the relevant set of distributions.  Recall from \cref{main} that $D$ is the difference between the largest and smallest fixed points of $F^2$ (i.e.\ roots of $F^2(x)-x$) in $[0,1]$.

Suppose for a contradiction that $D$ is not continuous at $\bp\in\cQ$, so that there exists a continuous family $(\bp(t))_{t\in[0,1]}$ in $\cQ$ with $\bp(t)\to\bp(0)=\bp$ but $D(\bp(t))\not\to D(\bp)$ as $t\to 0$.  The complex roots of a polynomial vary continuously with its coefficients (possibly becoming or ceasing to be coincident, and going off to or arriving from infinity).  Therefore, either some root of $F_{\bp(t)}^2(x)-x$ must enter the interval $[0,1]$ at $t=0$, or some complex root must become real.

The first possibility is ruled out because $F_\bp^2(0)=F(1-p_0)\neq 0$ and $F_\bp^2(1)=1-p_0\neq 1$, so the polynomial $F_\bp^2(x)-x$ does not have roots at $0$ or $1$.  Turning to the second possibility, since the polynomial $F^2(x)-x$ has real coefficients, any non-real roots come in conjugate pairs, so such a pair must become coincident and real at $t=0$, so $F_\bp^2(x)-x$ has a repeated root in $[0,1]$.  Now recall \cref{roots-in-pairs}, and note that the special root $x^*=x^*(\bp(t))$ varies continuously with $t$.  It is possible for two roots to become coincident and real and simultaneously coincide with $x^*$ (as indeed happens in many cases), but this would not account for the discontinuity in $D$.  Hence the polynomial $F_\bp^2(x)-x$ must have a repeated root in $[0,1]$ that is \emph{not} at $x^*$.  But then by \cref{roots-in-pairs} it must have \emph{another} repeated root in $[0,1]$.  Hence there are at least $5$ roots in $[0,1]$, counted with multiplicity.  (Essentially, the picture must resemble Figure~\ref{discontinuousnormalfigure}.)  But $G_\bp$ is (at most) a quadratic, so $F_\bp^2(x)-x$ is (at most) a quartic, a contradiction.
\end{proof}

We break the proof of \cref{cont} (iii) into parts.

\begin{prop}[Forcing strategy]\label{p1muprop}
Consider the escape game, and let $\mu$ be the mean of the offspring distribution $\bp$.
If $p_1\mu > 1$, then $\ei>0$.
\end{prop}

\begin{proof}
We give two explanations, one analytic and one in terms of the game.
First, $\ei$ is the largest solution in $[0,1]$ of $F(H(x))-x=0$.
The function $F(H(x))$ is continuous on $[0,1]$, with
$F(H(0))=0$ and $F(H(1))-1<0$.
Further one can calculate that the derivative of $F(H(x))-x$ at $x=0$
is $p_1\mu-1$. Hence if $p_1\mu>1$, there must be a solution
to $F(H(x))-x=0$ somewhere in $(0,1)$, and hence $\ei>0$.

For the alternative argument, consider the set of paths in the tree $\T$, starting at the root, with the property that every vertex at odd depth on the path has precisely one child.  The
union of these paths is a subtree $\T'$ containing the root.  Each odd-depth vertex of $\T'$ has exactly two neighbors: its parent and its unique child.  Let $\T''$ be the tree obtained by removing every odd-depth vertex from $\T'$ and joining its parent directly to its child.
Now $\T''$ is a Galton-Watson tree whose offspring distribution is the original distribution $\bp$ thinned by $p_1$ (i.e., conditional on a random variable $M$ distributed according to $\bp$, the number of offspring of a vertex is Binomial$(M,p_1)$).  Therefore if $p_1\mu>1$ then with positive probability $\T''$ is infinite.  On that event, Escaper can win the escape game on the original tree $\T$, provided he moves first, by always playing in $\T'$, so that Stopper never has any choice.
\end{proof}

\begin{prop}[Perturbation]\label{escapeptthm}
Let $\cS:=\{\bp: \ei=0\}$ be the set of distributions with zero probability of an Escaper win.
If $\bp\in\cS$ with $p_1\mu<1$, then $\cS$ also contains a neighbourhood of
$\bp$ in the metric space $(M_1,d_1)$.
\end{prop}
\begin{proof}
A distribution is in $\cS$ if and only if there is no root of $F(H(x))-x$ in $(0,1]$.  (There is always a root at $0$.)
Let $\bp\in\cS$ with $p_1\mu<1$.

The derivative of $F(H(x))-x$ is $F'(H(x))H'(x)-1$, which equals $p_1\mu-1$ at $x=0$.
And by continuity of the generating function, the derivative converges to $p_1\mu-1$ as $x\downarrow 0$.
Let $\tbp\in\cS$ be another distribution with corresponding functions $\tF$ and $\tH$.
Then we have $|F(x)-\tF(x)|\leq d_1(\tbp, \bp)$ and $F'(x)-\tF'(x)|\leq d_1(\tbp, \bp)$ for all $x\in[0,1]$, and similarly for $H$ and $\tH$.

Putting these facts together, for any $\epsilon>0$,
there exist $u$ and $\delta_1$ such that if $x\in[0,u]$ and $d_1(\bp, \tbp)<\delta_1$
then
\[
\left| \tF'(\tH(x)) \tH'(x) - \mu_1 p \right| <\epsilon.
\]
Hence by choosing $\epsilon$ small enough, we have that the derivative
of $\tF(\tH(x))-x$ is negative on all of $[0,u]$. Since $\tF(\tH(0))-0=0$,
it follows that $\tF(\tH(x))-x$ has no roots on $(0,u]$.

Now $F(H(x))-x$ is negative on all of $[u,1]$, and so (by uniform continuity
on closed intervals) is bounded away from 0 on that interval.
We have $|\tF(\tH(x))-F(H(x))| \leq d_1(\tbp, \bp)^2$.
So we can find $\delta_2>0$ such that
if $d_1(\bp, \tbp)<\delta_2$ then $\tF(\tH(x))-x$ has no roots on $[u,1]$.

Taking $\delta=\min(\delta_1, \delta_2)$, we find that if $d_1(\bp, \tbp)<\delta$
then $\tbp\in\cS$, as required.
\end{proof}

\begin{proof}[Proof of \cref{cont} (iii)]
This is immediate from \cref{p1muprop,escapeptthm}.
\end{proof}


\section{Length of the game}\label{sec:length}

In this section we prove Theorem \ref{thm:length}.
Initially we write the proof for the case of the normal game,
and indicate the analogous argument for the case of the
mis\`ere game at the end.

The function $F$ is strictly decreasing with $F(0)>0$
and $F(1)=0$, so has a unique fixed point. We begin by considering
the derivative of $F$ and related functions at this fixed point.
\begin{lemma}\label{lemma:Dcriteria}
Let $x^*$ be the unique fixed point of $F$.
\begin{itemize}
\item[(a)]If $\bp\in\cB$, then $F'(x^*)\leq -1$.
\end{itemize}
More precisely:
\begin{itemize}
\item[(b)]If $\bp\in\cB^o$, then $F'(x^*)<-1$;
\item[(c)]If $\bp\in\cB\intersect\partial\cB$, then $F'(x^*)=-1$.
\end{itemize}
\end{lemma}

Note that since $F(x)=1-G(x)$, we have $F'(x)=-G'(x)$.
We can also rewrite $F'(x^*)$ in terms of the function
$F^2(x)-x$ which we plotted for example in Figure
\ref{binarynormalfigure} and
Figure \ref{discontinuousnormalfigure}.
Writing $\Delta(x)=\frac{d}{dx}(F^2(x)-x)$
as in the proof of Lemma \ref{roots-in-pairs}
we have $\Delta(x^*)=F'(F(x^*))F'(x^*)-1=F'(x^*)^2-1$.
Hence
\begin{align}
\label{FDeltaequivalence}
F'(x^*)<-1 \,\, &\Leftrightarrow \,\, \Delta(x^*)>0\\
\nonumber
F'(x^*)=-1 \,\, &\Leftrightarrow \,\, \Delta(x^*)=0
\end{align}

Before proving Lemma \ref{lemma:Dcriteria},
we note a useful technical property:
\begin{lemma}\label{lemma:usefulp}
Let $x^*\in(1/2,1)$. Then there is an offspring distribution
$\hat{\bp}$ with generating function $\hat{G}$ satisfying
$\hat{G}(x^*)=1-x^*$ and $\hat{G}'(x^*)>1$.
\end{lemma}
\begin{proof}
We have
\begin{equation}\label{line1}
x^*>1-x^*,
\end{equation}
but $(x^*)^k <1-x^*$ for sufficiently large $k$.
Hence there is $k\geq 2$ such that
\begin{align}
\label{linekminus}
(x^*)^{k-1}&\geq 1-x^*\\
\intertext{and}
\label{linek}
(x^*)^{k}&< 1-x^*.
\end{align}

Then from (\ref{line1}) and (\ref{linek}),
for some $q\in(0,1)$,
the generating function $\hat{G}(x)=(1-q)x+qx^k$
(corresponding to the distribution $\hat{\bp}$ with $\hat{p}_1=1-q$
and $\hat{p}_k=q$)
has
$\hat{G}(x^*)=1-x^*$.

Also (\ref{linekminus}) gives
\begin{align*}
x^*&\geq 1-(x^*)^{k-1},
\\
\intertext{so that}
(x^*)^{k-1}&\geq \Big[1-(x^*)^{k-1}\Big]^{k-1}\\
&> 1-(k-1)(x^*)^{k-1},
\end{align*}
and so
$k(x^*)^{k-1}>1$.
Then $\hat{G}'(x^*)=1-q+qk(x^*)^{k-1}>1$ as required.
\end{proof}

\begin{proof}[Proof of Lemma \ref{lemma:Dcriteria}]
The proof of part (a) is very easy.
Note that the function $F^2(x)-x$ is positive at $x=0$,
is negative at $x=1$, and is zero at $x=x^*$.
If in addition $F'(x^*)>-1$ then by (\ref{FDeltaequivalence}),
$F^2(x)-x$ crosses from negative to positive at $x=x^*$,
and so must have at least one fixed point in $(0,x^*)$ and
another in $(x^*, 1)$. Then by Theorem \ref{main}, $D>0$.
Hence if $\bp\in\cB$ (i.e.\ if $D=0$) we must indeed have
$F'(x^*)\leq -1$.

For part (b),
suppose indeed that $F'(x^*)=-1$, i.e.\ $G'(x^*)=1$.
We will show that
$\bp$ is not in $\cB^o$, by showing
that there are points of $\cB^c$ arbitrarily close to $\bp$.

First note that we must have $x^*>1/2$
(excluding the trivial case $G(x)\equiv x$, i.e. $p_1=1$,
where $x^*=1/2$), since by strict convexity of $G$,
\begin{align*}
1&=G(1)\\
&> G(x^*)+(1-x^*)G'(x^*)\\
&=1-x^*+(1-x^*)\times 1\\
&=2(1-x^*).
\end{align*}
So from Lemma \ref{lemma:usefulp}, there is an offpsring distribution $\hat{\bp}$ whose generating function
$\hat{G}$ has
$\hat{G}(x^*)=1-x^*$ and $\hat{G}'(x^*)>1$.
Then for any $\epsilon>0$, the
distribution $\bp_{\epsilon}
:=(1-\epsilon)\bp+\epsilon\hat{\bp}$
with generating function
\begin{equation}\label{Gepsilon}
G_\epsilon(x)=(1-\epsilon)G(x)+\epsilon\hat{G}(x)
\end{equation}
also has $G_\epsilon(x^*)=1-x^*$ and $G_\epsilon'(x^*)>1$.
Hence by part (a),
for all $\epsilon$, $\bp_\epsilon\notin\cB$.
But since $\bp_\epsilon$ is arbitrarily close to $\bp$
in $M_0$, we have that $\bp\notin\cB^0$, as required for part (b).

Finally for part (c), suppose that $\bp\in\cB$
with $F'(x^*)>-1$. We need to show that $\bp\in\cB^o$,
i.e.\ that all distributions in some neighbourhood of $\bp$ in $M_0$ also have no draws.

The function $F(F(x))-x$ has a unique zero at $x^*$,
and has derivative $\Delta(x)=F'(F(x))F'(x)-1$ which is continuous on $(0,1)$
with $\Delta(x^*)<0$, as at using (\ref{FDeltaequivalence}).
Hence for some $\epsilon>0$,
\begin{equation}\label{FF1}
\frac{d}{dx}\big(F(F(x))-x\big)<-\epsilon \text{ for all }
x\in [x^*-\epsilon, x^*+\epsilon].
\end{equation}
Also $F(F(x))-x$ is a continuous function and so attains its
bounds on any closed interval; hence for some $\delta>0$,
\begin{equation}\label{FF2}
\big| F(F(x))-x \big| > \delta \text{ for all }
x\in [0,x^*-\epsilon/2]\union[x^*+\epsilon/2,1].
\end{equation}
We want to show that properties like (\ref{FF1}) and (\ref{FF2})
continue to hold if we perturb $\bp$ slightly.

We note the following properties:
\begin{itemize}
\item[(i)]
$F$ is uniformly continuous on $[0,1]$.
\item[(ii)]
For any $x$,
the quantity $F(x)$ is continuous as a function of $\bp$,
uniformly in $x$; specifically, for all $\bp$, $\tildebp$, and $x$,
\[
\Big|F_{\bp}(x)-F_{\tildebp}(x)\Big|\leq d_0(\bp, \tildebp).
\]
\end{itemize}
Combining (i) and (ii) with (\ref{FF2}), it follows
that whenever $d_0(\bp,\tildebp)$ is sufficiently small,
(\ref{FF2}) again holds with $F$ replaced by $F_{\tildebp}$
and $\delta$ by $\delta/2$.

Continuing, note that:
\begin{itemize}
\item[(iii)]
The function $F$ maps $[x^*-\epsilon/2, x^*+\epsilon/2]$ to some $[a,b]$ with $0<a<b<1$.
\item[(iv)]
$F'$ is uniformly continuous on $[0,z]$, for any $z<1$;
specifically, for all $0<x<y$,
\[
\big|F'(x)-F'(y)\big|\leq \sum_{n=2}^\infty n(y^{n-1}-x^{n-1})
\leq |y-x|\sum_{n=2}^\infty n^2 z^{n-2}.
\]
\item[(v)]
For any given $x$, $F'(x)$ is continuous as a function of
$\bp$; specifically, for all $\bp$, $\bp'$, and $x$,
\[
\big|F_{\bp}'(x)-F_{\tildebp}'(x)\big|\leq  d_0(\bp, \tildebp)\sum_{n=2}^\infty nx^{n-1}.
\]
\end{itemize}
Combining (i)-(v) with (\ref{FF1}),
and using $\frac{d}{dx}(F(F(x))-x)=F'(F(x))F'(x)-1$,
it follows
that whenever $d_0(\bp, \tildebp)$ is sufficiently small,
(\ref{FF1}) holds
with $F$ replaced by $F_{\tildebp}$ and
$\epsilon$ replaced by $\epsilon/2$ throughout.

The new versions of (\ref{FF1}) and (\ref{FF2}) thus obtained then guarantee that for all $\tildebp$ in some
neighbourhood of $\bp$ in $M_0$,
the function $F^2_{\tildebp}$ has no fixed point outside
$[x^*-\epsilon/2, x^*+\epsilon/2]$, and has at most
one fixed point inside that interval. Hence by Theorem
\ref{main}, the game with distribution $\bp'$ has no draws.
This shows that $\bp$ is in the interior of $\cB$, as required
for (c).
\end{proof}

\begin{proof}[Proof of Theorem \ref{thm:length}(i)]
We wish to show that if $\bp\in\cB^{o}$, then $\E T<\infty$
(and then certainly $\E T^*<\infty$ also since $T^*\leq T$).

Note that $\P(T>n)$ is the probability that neither player
can force a win within $n$ moves, which is $D_n$. Hence
\begin{equation}\label{ET}
\E T=\sum_{n\geq0}\P(T>n)=\sum_{n\geq0} D_n=\sum_{n\geq0} \big[(1-P_n)-N_n].
\end{equation}
Any game won by the first player has odd length, and any
game won by the second player has even length.
Then as in the proof of Theorem \ref{main}, we have
\[
1-P_{2k-1}=1-P_{2k}=F^{2k}(1)
\text{\,\,\, and \,\,\,}
N_{2k}=N_{2k+1}=F^{2k}(0).
\]
Since $\bp\in\cB$, Theorem \ref{main} gives that $F^2$
has a unique fixed point which is $x^*$,
and since $\bp\in\cB^o$, Lemma
\ref{lemma:Dcriteria} gives that $|F'(x^*)|<1$.
So both $1-P_n$ and $N_n$ converge
exponentially quickly to $x^*$.
Hence the sequence $D_n=1-P_n-N_n$ has finite sum,
and (\ref{ET}) gives $\E T<\infty$ as required.
\end{proof}

The following simple result must be well known, but
we don't have a precise reference:
\begin{lemma}\label{lemma:GWheight}
A Galton-Watson process with mean offspring size $1$
has infinite expected height.
\end{lemma}
\begin{proof}
Let $V$ be the height of the process, and
$a_n=P(V\geq n)$ the probability that the process
survives at least to height $n$,
so that $\E V=\sum a_n$.
Then for example by
conditioning on the number of children of the root in
a standard way, we have $a_{n+1}=1-G(1-a_n)$.

Note that $G'(1)=1$, so that as $x\uparrow 1$, Taylor's
theorem gives
\[
G(x)=1-(1-x)+O(1-x)^2,
\]
so that as $y\downarrow 0$,
\[
1-G(1-y)-y=O(y^2).
\]
As $n\to\infty$ we have $a_n\to0$, and so
\[
a_{n+1}-a_n=O(a_n^2).
\]
In particular, for some constant $c$, for large enough $n$
(say $n\geq n_0)$
\begin{equation}\label{recurrencebound}
a_{n+1}>(1-ca_n)a_n.
\end{equation}
Since $a_n\to 0$, it follows from (\ref{recurrencebound})
that $\prod_{n=n_0}^\infty (1-ca_n)=0$.
But this is equivalent to $\sum a_n=\infty$.
Hence $\E V=\infty$, as required.
\end{proof}

\begin{proof}[Proof of Theorem \ref{thm:length}(ii)]
We assume $\bp\in\cB\intersect\partial\cB$.
So the probability of a draw is $0$, and
$N$, the probability of a first-player win,
is equal to $x^*$, the unique fixed point of $F$.
Also, by Lemma \ref{lemma:Dcriteria}, $G'(N)=-F'(N)=1$.

We may mark each node of the tree as an $\cN$-node
(a first-player win), or a $\cP$-node
(a second-player win). The root is an $\cN$-node with probability $N$ and a $\cP$-node with probability $P=1-N$.

With these marks we can see the tree as a two-type
branching process. Each $\cP$-node has only $\cN$-type
children. Conditional on being a $\cP$-node, the
number of children has probability mass function
$\tp_k, k\geq 0$ given by $\tp_k=p_k N^k/P$, with mean
\begin{equation}
\label{meanNchildren}
\tilde{\mu}=\sum k p_k N^k/P = NG'(N)/P.
\end{equation}
Each $\cN$-node has at least one $\cP$-type child.
Conditional on being a $\cN$-node, the probability
of having precisely one $\cP$-type child is
$\hat{p}_1$ given by
\begin{equation}
\label{hatp1}
\hat{p}_1=\sum_{j=1}^\infty p_j jPN^{j-1}/N = PG'(N)/N.
\end{equation}

Now we define a reduced subtree. Call an $\cN$-node
bad if it is the child of another $\cN$-node.
(Such a node is never part of an optimal line of play.)
Call a $\cP$-node bad if it is the child of a node which
has another $\cP$-type child. (The winning player can
guarantee to win without visiting this node.)

Remove all the bad nodes, and consider the reduced tree
consisting of all those nodes still connected to the root.
A node $v$ is in the reduced tree if the player without a winning strategy can guarantee either to win
or to visit $v$ (as in the definition of the quantity $T^*$).
In particular, $T^*$ is the height of this reduced tree.

The reduced tree is a two-type Galton-Watson process.
Suppose that the root is a $\cP$-node. Then all
the nodes at even levels are $\cP$-nodes, and all the nodes
at odd levels are $\cN$-nodes.
The expected number of grandchildren of the root
in this reduced tree is the product
$\tilde{\mu}\hat{p}_1$
of (\ref{meanNchildren})
and (\ref{hatp1}); this product is $G'(N)^2$ which
equals $1$. If we consider only the nodes at even levels,
we obtain a simple Galton-Watson process, with mean
offspring size $1$, and hence (by Lemma \ref{lemma:GWheight})
with infinite expected height. This
gives $\E T^*=\infty$ as required.
\end{proof}

\begin{proof}[Proof of Theorem \ref{thm:length}(iii) and (iv)]
As at (\ref{ET}), $\E T=\sum_{n\geq 0} D_n$,
where $D_n$ is the probability that neither player can force
a win within $n$ turns.

Suppose that $\bp^{(m)}, m\geq 1$
is a sequence of offspring distributions
in $\cB$
converging in $M_0$ as $m\to\infty$ to a
distribution $\bp^{(\infty)}\in\partial\cB$.
Write $\E^{(m)}$ and $\E^{(\infty)}$
for expectations in the models corresponding to $\bp^{(m)}$ and $\bp^{(\infty)}$ respectively, and similarly $D^{(m)}_n$ and $D^{(\infty)}_n$ for the draw probabilities.

We have $E^{(\infty)} T=\sum_{n\geq 0} D^{(\infty)}_n=\infty$
(this follows from
Theorem \ref{thm:length}(ii) in the case $\bp^{(\infty)}\in
\cB\intersect\partial\cB$, and from the fact that
$D^{(\infty)}_n\not\to0$ in the case $\bp^{(\infty)}\in\cB
\setminus\partial\cB$).

For any fixed $n$ we have $D_n^{(m)}\to D_n^{(\infty)}$ as $m\to\infty$,
since the distribution of the first $n$ levels of the tree under
$\bp^{(m)}$ converges in total variation distance to
the distribution under $\bp^{(\infty)}$.
So $\liminf_{m\to\infty}\sum_{n=0}^\infty D_n^{(m)}
\geq \sum_{n=0}^K D_n^{(\infty)}$ for any $K$.
We can make this lower bound arbitrarily large by taking
$K$ large enough, since $\sum_{n=0}^\infty D_n^{(\infty)}=\infty$.
So indeed $\E^{(m)}T=\sum_{n=0}^\infty D_n^{(m)}\to\infty$ as $m\to\infty$, as required.

Finally suppose that the limiting distribution
$\bp^{(\infty)}$ is in $\cB\intersect\partial\cB$.
To show that the mean of $T^*$ tends to infinity, we
apply a similar argument but now to the reduced tree
constructed in the proof of Theorem \ref{thm:length}(ii)
above.

Write $x^*(\bp)$ for the unique fixed point
of the function $F_{\bp}$.
We have $x^*(\bp^{(m)})\to x^*(\bp^{(\infty)})$
as $m\to\infty$ (since each $F_{\bp^{(m)}}$
and $F_{\bp^{(\infty)}}$ is continuous and strictly
decreasing, and $F_{\bp^{(m)}}(x)\to F_{\bp^{(\infty)}}(x)$
as $m\to\infty$ uniformly over $x\in[0,1]$).
So the first-player and second-player win probabilities
$N=x^*$ and $P=1-x^*$ also converge to their
values under $\bp^{(\infty)}$
(since the draw probability $D$ is $0$ in each case).

It follows that for any $n$,
the distribution of the first $n$ levels of the reduced tree
under $\bp^{(m)}$ converges in total variation distance
as $m\to\infty$ to the distribution under $\bp^{(\infty)}$.
Recall that $T^*$ is the height of this reduced tree.
Under the limit distribution,
we have $\E^{(\infty)} T^*=\sum_{n\geq 0} \P^{(\infty)}(T^*>n)=\infty$,
by Theorem \ref{thm:length}(ii), and by an analogous argument
to the one above for $\E T$, we obtain $\E^{(m)} T^*\to\infty$
as $m\to\infty$ as required.
\end{proof}

We have completed the proof of Theorem \ref{thm:length}
for the case of the normal game. One can prove the result
for the mis\`ere case in an entirely similar way, which
we indicate briefly.

Let $\tx^*$ be the unique fixed point of the function $H$.
For the mis\`ere case, we have analogous criteria
to those in Lemma \ref{lemma:Dcriteria} with $x^*$
replaced by $\tx^*$ (note that $H'\equiv F'$).

In the proof of Lemma \ref{lemma:Dcriteria},
we relied on the fact that $x^*>1/2$ in order
to apply Lemma \ref{lemma:usefulp}.
Since $H$ and $F$ are both decreasing functions,
and $H\geq F$, we have $\tx^*>x^*$ so again $\tx^*>1/2$.
Just as at (\ref{Gepsilon}), if we have a distribution
$\bp$ with generating function $G$ such that $H(\tx^*)=
1-G(\tx^*)+G(0)=\tx^*$ and $H'(\tx^*)=-1$, we can obtain a distribution $\bp_\epsilon$ which is arbitrarily close to $\bp$ in $M_0$, with
generating function $G_\epsilon$
such that $H_\epsilon(\tx^*)=1-G_\epsilon(\tx^*)+G_\epsilon(0)=\tx^*$,
and $H'_\epsilon(\tx^*)<-1$. The rest of the
argument goes through identically, with
$\tx^*$ replacing $x^*$ and
$H$ replacing $F$ throughout.

For the proof of Theorem \ref{thm:length}(ii) in the mis\`ere
case, we again consider a two-type Galton-Watson tree,
where each node is an $\vnm$-node (a first-player win for the
mis\`ere game) or a $\vpm$-node (a second-player win for the
mis\`ere game). The root is an $\vnm$-node
with probability $\tN=\tx^*$ and a $\vpm$-node with probability
$\tP=1-\tx^*$.

Again each $\vpm$-node has only $\vnm$-type children.
Conditional on being a $\vpm$-node, the number of children
has mean $\tN G'(\tN)/\tP$ just as in (\ref{meanNchildren}).
In the mis\`ere case, each $\vnm$ node either has at least
one $\vpm$-type child, or has no children at all. Just as in
(\ref{hatp1}), conditional on being a $\vnm$-node,
the probability of having precisely one $\vpm$-type child
is $\tP G'(\tN)/\tN$. The product of these two
quantities is $G'(\tN)^2$ which again equals $1$.
The rest of the proof is entirely analogous.

\section*{Acknowledgements}

A substantial part of this research was done during an extended visit by both authors to the Mittag-Leffler Institute.  We thank the Institute for the hospitality and excellent facilities.  We thank Johan W\"astlund for valuable conversations.

\bibliography{games}

\begin{thebibliography}{10}

\bibitem{undir}
R.~Basu, A.~E. Holroyd, J.~B. Martin, and J.~W\"astlund.
\newblock Trapping games on random boards.
\newblock {\em Ann. Appl. Probab.}, 26(6):3727--3753, 2016.

\bibitem{maker-breaker-geometric}
A.~Beveridge, A.~Dudek, A.~Frieze, T.~Müller, and M.~Stojaković.
\newblock Maker-breaker games on random geometric graphs.
\newblock {\em Random Structures Algorithms}, 45(4):553--607, 2014.

\bibitem{BroutinDevroyeFraiman}
N.~Broutin, L.~Devroye, and N.~Fraiman.
\newblock Recursive functions on conditional {G}alton--{W}atson trees.
\newblock 2018.
\newblock arXiv:1805.09425.

\bibitem{frogs}
M.~Deijfen, A.~E. Holroyd, and J.~B. Martin.
\newblock Friendly frogs, stable marriage, and the magic of invariance.
\newblock {\em Amer. Math. Monthly}, 124(5):387--402, 2017.

\bibitem{Dekking}
F.~M. Dekking.
\newblock Branching processes that grow faster than binary splitting.
\newblock {\em Amer.\ Math.\ Monthly}, 98(8):728--731, 1991.

\bibitem{biased}
A.~Ferber, R.~Glebov, M.~Krivelevich, and A.~Naor.
\newblock Biased games on random boards.
\newblock {\em Random Structures Algorithms}, 46(4):651--676, 2015.

\bibitem{phd}
A.~E. Holroyd.
\newblock {\em Percolation Beyond Connectivity}.
\newblock PhD thesis, University of Cambridge, 2000.

\bibitem{logic}
A.~E. Holroyd, A.~Levy, M.~Podder, and J.~Spencer.
\newblock Second order logic on random rooted trees.
\newblock 2017.
\newblock arXiv:1706.06192.

\bibitem{dir}
A.~E. Holroyd, I.~Marcovici, and J.~B. Martin.
\newblock Percolation games, probabilistic cellular automata, and the hard-core
  model.
\newblock arXiv:1503.05614. To appear in \textit{Probab.\ Thy.\ Related
  Fields.}

\bibitem{Janson2009}
S.~Janson.
\newblock On percolation in random graphs with given vertex degrees.
\newblock {\em Electron. J. Probab.}, 14:86--118, 2009.

\bibitem{JohnsonPodderSkerman}
T.~Johnson, M.~Podder, and F.~Skerman.
\newblock Random tree recursions: which fixed points correspond to tangible
  sets of trees?
\newblock 2018.
\newblock arXiv:1808.03019.

\bibitem{KarpSipser}
R.~M. Karp and M.~Sipser.
\newblock Maximum matching in sparse random graphs.
\newblock In {\em Foundations of Computer Science, 1981. SFCS'81. 22nd Annual
  Symposium on}, pages 364--375. IEEE, 1981.

\bibitem{konig}
D.~K{\"o}nig.
\newblock {\"U}ber eine {S}chlussweise aus dem {E}ndlichen ins {U}nendliche.
\newblock {\em Acta Sci. Math. (Szeged)}, 3(2-3):121--130, 1927.

\bibitem{MartinStasinski}
J.~B. Martin and R.~Stasi{\'n}ski.
\newblock Minimax functions on {G}alton-{W}atson trees.
\newblock 2018.
\newblock arXiv:1806.07838.

\bibitem{StojakovicSzabo2005}
M.~Stojakovi{\'c} and T.~Szab{\'o}.
\newblock Positional games on random graphs.
\newblock {\em Random Structures Algorithms}, 26(1-2):204--223, 2005.

\bibitem{wastlund}
J.~W{\"a}stlund.
\newblock Replica symmetry of the minimum matching.
\newblock {\em Ann. of Math. (2)}, 175(3):1061--1091, 2012.

\end{thebibliography}
\bibliographystyle{abbrv}

 \end{document}